\documentclass{article}
\usepackage{amssymb,amsmath,amsthm, hyperref}
\newtheorem{theorem}{Theorem}
\newtheorem{definition}{Definition}
\newtheorem{lemma}{Lemma}
\newtheorem{proposition}{Proposition}
\newtheorem{corollary}{Corollary}
\newtheorem{remark}{Remark}
\date{}
\numberwithin{equation}{section} \numberwithin{theorem}{section}
\numberwithin{lemma}{section} \numberwithin{corollary}{section}
\numberwithin{remark}{section} \numberwithin{proposition}{section}
\numberwithin{definition}{section}

\newcommand{\n}{\noindent}
\newcommand{\vs}{\vskip}
\newcommand{\Div}{\operatorname{div}}
\newcommand{\loc}{\operatorname{loc}}

\begin{document}

\title{On the Regularity of the Free Boundary for Quasilinear Obstacle Problems}

\author{S. Challal$^1$, A. Lyaghfouri$^{1,2}$,  J. F. Rodrigues$^3$ and R. Teymurazyan$^3$\\
\\
$^1$ Glendon College, York University\\
Toronto, Ontario, Canada\\
\\
$^2$ Fields Institute, 222 College Street\\
Toronto, Ontario, Canada\\
\\
$^3$ University of Lisbon/CMAF\\
Lisbon, Portugal} \maketitle

\begin{abstract}
We extend basic regularity of the free boundary of the obstacle problem to some
classes of heterogeneous quasilinear elliptic operators with variable growth that includes, in particular,
 the $p(x)$-Laplacian. Under the assumption of Lipschitz continuity
  of the order of the power growth  $p(x)>1$, we use the growth rate of the solution
  near the free boundary to obtain its porosity, which implies that the free boundary is of Lebesgue
  measure zero for $p(x)$-Laplacian type heterogeneous obstacle problems. Under additional assumptions
  on the operator heterogeneities and on data we show, in two different cases, that up to a negligible singular
  set of null perimeter the free boundary is the union of at most a countable family of $C^1$ hypersurfaces:
  $i)$ by extending directly the finiteness of the $(n-1)$-dimensional Hausdorff measure of the free boundary to the case of
  heterogeneous $p$-Laplacian type operators with constant $p, 1<p<\infty$; $ii)$ by proving the
  characteristic function of the coincidence set is of bounded variation in the case of
  non degenerate or non singular operators with variable power growth $p(x)>1$.
\end{abstract}

\section{Introduction}

\quad
In $\cite{[C]}$ Caffarelli remarked that the quadratic growth of the solution from the free boundary of the obstacle problem for the Laplacian implies an estimate of the $(n-1)$-dimensional Hausdorff  ($\mathcal{H}^{n-1}$) measure of the free boundary and a stability property. This result has a simple generalization to second order linear elliptic operators with Lipschitz continuous coefficients and regular obstacles, as observed by one of the authors in $\cite{[R1]}$, page 221. This generalization allows the extension of those properties to the free boundaries of $C^{1,1}$ solutions of the obstacle problem for certain quasilinear operators of minimal surfaces type (see Theorem 7:5.1 of  $\cite{[R1]}$, page 246). These results are important since they are first steps for the higher regularity of the free boundary in obstacle-type problems (see the recent monograph $\cite{[PSU]}$ for problems with Laplacian).

In an earlier work $\cite{[BK]}$ in the framework of homogeneous non degenerate quasilinear operators that allow solutions to the obstacle problem with bounded second order derivatives, Br\'ezis and Kinderlehrer have obtained the first result on the regularity of the free boundary in any spatial dimension: under a natural nondegeneracy condition on the data, the coincidence set of the solution with the obstacle has locally finite perimeter (see Corollary 2.1 of $\cite{[BK]}$). As an important consequence, by a well-known result of De Giorgi (see $\cite{[G]}$, page 54), the free boundary  $\partial\{u>0\}$ may be written, up to a possible singular set of null perimeter (i.e. of $\|\nabla \chi_{\{u>0\}}\|$-measure zero) as a countable union of $C^1$ hypersurfaces.

On the other hand, it was shown by Karp, Kilpel\"{a}inen, Petrosyan and Shahgholian $\cite{[KKPS]}$, for the $p$-obstacle problem, with constant $p, 1<p< \infty$, that the free boundary is porous with a certain constant $\delta>0$, that is, there exists $r_0 >0$ such that for each $x\in \partial\{u>0\}$ and $0<r<r_0$, there exists a point $y$ such that $B_{\delta r}(y)\subset B_{r}(x)\setminus\partial\{u>0\}$. The porosity of the free boundary is a consequence of the controlled growth of the solution from the free boundary. This interesting property was also established in \cite{[CL2]} in the $p(x)$-Laplacian framework and is now extended here to the more general class of heterogeneous quasilinear degenerate elliptic operators in Sobolev spaces of variable exponent $p(x), 1<p(x)<\infty$.

However, porosity is only a first step in the regularity of the free boundary and, for instance, does not prevent it of being a Cantor-type subset. But since a porous set in $\mathbb{R}^n$ has Hausdorff dimension strictly smaller that $n$ (see $\cite{[MV]}$ or $\cite{[Z]}$), it follows that the free boundary has Lebesgue measure zero, which allows us to write the solution of the obstacle problem as an a.e. solution of a quasilinear elliptic equation in the whole domain involving the characteristic function $\chi_{\{u>0\}}$ of the non-coincidence set (see Theorem 3.1 below, that extends earlier results in $\cite{[CL1]}$ and $\cite{[CL2]}$, respectively, for the $A$-obstacle and $p(x)$-obstacle problems). This property is important to show, under general nondegeneracy assumptions on the data, the stability of the non-coincidence set in Lebesgue measure as a consequence of the continuous dependence of their characteristic functions. As a consequence of our results, we can extend this property to more general quasilinear obstacle problems, including for instance, Corollary 1.1 of \cite{[CLR]}, Theorem 4 of \cite{[R]} and Theorem 2.8 of \cite{[RSU]}.

Hausdorff measure estimates were obtained directly for homogeneous nonlinear operators of the $p$-obstacle problem ($2<p<\infty$) by Lee and Shahgholian $\cite{[LS]}$, for general potential operators by Monneau $\cite{[M]}$ in a special case corresponding to an obstacle problem arising in superconductor modelling with convex energy, and by three of the authors in \cite{[CLR]} to the so called A-obstacle in Orlicz-Sobolev spaces, that includes a class of degenerate and singular elliptic operators larger than the $p$-Laplacian ($1<p<\infty$). Essentially with similar estimates obtained in \cite{[CLR]}, the later work $\cite{[ZZ12]}$ reobtained the same results for a slightly different class of homogeneous quasilinear elliptic operators that includes also the $p$-Laplacian case.

As it is well-known from geometric measure theory, the importance of the estimate on the $(n-1)$-dimensional Hausdorff measure of the free boundary lies in the fact that, by a result of Federer, it implies that the non-coincidence set $\{u>0\}$ is a set of locally finite perimeter.  A main result of our present work is the extension of properties on the  $\mathcal{H}^{n-1}$-measure of the free boundary to a more general class of heterogeneous quasilinear elliptic operators which includes a non degenerate variant of the $p(x)$-Laplacian and extensions of the heterogeneous $p$-Laplacian with $1<p<\infty$ constant. The first result, following the Br\'ezis and Kinderlehrer approach, will be a consequence of the new result, even for linear operators, on the local bounded variation of the coincidence set in the heterogeneous obstacle problem. By well known results, the estimate on the perimeter of the (free) boundary is equivalent to the  $\mathcal{H}^{n-1}$-measure of the essential (free) boundary, which is also called the measure-theoretic (free) boundary (see $\cite{[EG]}$, page 208). The free boundary points that are not in the essential free boundary have $\|\nabla \chi_{\{u>0\}}\|$-measure zero or, equivalently, null perimeter. In the second case of a possibly degenerate or singular heterogeneous operator with $p$ constant we extend the Caffarelli direct approach following the developments of $\cite{[LS]}$ and \cite{[CLR]}. However, we were unable to prove this for the case of the $p(x)$-obstacle problem, though we conjecture its essential free boundary has still finite $\mathcal{H}^{n-1}$-measure under similar assumptions.

Unlike the classical obstacle problem that admits $C^{1,1}$  solutions, where the extensions of the regularity of the free boundary from the Laplacian to the minimal surface type heterogeneous operators were simpler and did not require a new technique, the passage from the homogeneous case to the quasilinear heterogeneous obstacle problem raises several nontrivial difficulties. In particular, one has more a complicated form of the Harnack inequality, when we pass from the $p$-Laplacian to the variable $p(x)$-type operators, which seems is not applicable to the analysis of the free boundary regularity in the general framework that we now describe.

Let  $\Omega$ be  a bounded  open connected subset  of
$\mathbb{R}^{n}$,  $n \geqslant 2$, $f\in L^\infty  (\Omega) $
and $g\in W^{1,p(\cdot)}(\Omega) \cap L^\infty (\Omega)$, $g\geqslant 0$.
We consider the quasilinear obstacle problem ($a(\cdot)$-obstacle problem) with a zero obstacle:
$$
\begin{cases}
    & \displaystyle{A u:= \Div(a(x,\nabla u)) =
    f(x)}
      \quad \text{in}\quad \{u>0\}, \\
    &  u\geqslant 0 \quad \text{in }\quad \Omega,\\
    & u=g \quad  \text{on }\quad \partial\Omega,
\end{cases}
$$
where we denote by $\{u>0\}:=\{x\in\Omega:\quad u(x)>0\}$ the non-coincidence set.\\

The weak formulation  of this problem is given  by the
following  variational inequality
\begin{equation*}(P)
\begin{cases}
& \text{Find } u\in K_g
\text{ such that} :\\
&  \displaystyle{\int_\Omega }\Big(a(x,\nabla u)\cdot\nabla (v-u)+ f(x)(v-u)\Big)
  dx \,\geqslant\, 0
\qquad  \forall v \in K_g,
\end{cases}
\end{equation*}
where $~~K_g= \{ v\in  W^{1,p(\cdot)}(\Omega)~: ~~ v-g\in
W^{1,p(\cdot)}_0(\Omega),~~ v\geqslant 0 \quad \hbox{ a.e. in }
\Omega\}$, $p$ is a measurable real valued function defined in $\Omega$
and satisfying for some positive numbers $p_{-}$ and $p_{+}$

\begin{equation}\label{1.1}
1 < p_{-} \leqslant p(x) \leqslant p_{+}<\infty, \quad x\in \Omega.
\end{equation}\\

The space $W_0^{1,p(\cdot)}(\Omega)$ is defined as the closure of
$C_0^{\infty}(\Omega)$ in $W^{1,p(\cdot)} (\Omega)$, where $W^{1,p(\cdot)}(\Omega)$
is the variable exponent Sobolev space
$$
W^{1,p(\cdot)} (\Omega) = \Big\{u \in L^{p(\cdot)}\Omega)~:~ \nabla u \in
\big(L^{p(\cdot)}(\Omega)\big)^n\Big\}
$$
and $~~\displaystyle{L^{p(\cdot)} (\Omega) = \Big\{u : \Omega \rightarrow \mathbb{R} ~\mbox{ measurable}~:~\rho(u)=\int_{\Omega}
|u(x)|^{p(x)}\,dx < \infty~\Big\}}$

\n is equipped with the Luxembourg norm
$$
\|u\|_{L^{p(\cdot)}} = \inf\Big\{\lambda > 0~:~ \rho(u/\lambda)\leqslant 1~\Big\}.
$$
$W^{1,p(\cdot)}(\Omega)$ is equipped with the norm
$$
\|u\|_{W^{1,p(\cdot)}} = \|u\|_{L^{p(\cdot)}} + \|\nabla u\|_{L^{p(\cdot)}},
$$
where
$$
\|\nabla u\|_{L^{p(\cdot)}} = \displaystyle{\sum^{n}_{i=1}
\left\|\frac{\partial u}{\partial x_i}\right\|_{L^{p(\cdot)}}}.
$$

\vs 0,2cm \n By $B_r(x)$ we shall denote the open ball in $\mathbb{R}^n$ with center
$x$ and radius $r$. The conjugate of $p(x)$, defined by
${{p(x)}\over{p(x)-1}}$,  will be denoted by $q(x)$. If the center of a ball is not mentioned, then it is the origin.

We assume that the function $a:\Omega\times\mathbb{R}^n\rightarrow\mathbb{R}^n$ is such that $a(x,0)=0$ for
a.e. $x\in\Omega$, and satisfies the structural assumptions with $\kappa \in [0,1]$ and some positive constants $c_0$,
$c_1$, $c_2$, namely $\cite{[F]}$

\begin{equation}\label{1.2}
\sum_{i,j=1}^n\frac{\partial a_i}{\partial\eta_j}(x,\eta)\xi_i\xi_j\geq c_0\big(\kappa+|\eta|^2\big)^{\frac{p(x)-2}{2}}|\xi|^2,
\end{equation}
\begin{equation}\label{1.3}
\sum_{i,j=1}^n\bigg|\frac{\partial a_i}{\partial\eta_j}(x,\eta)\bigg|\leq c_1\big(\kappa+|\eta|^2\big)^{\frac{p(x)-2}{2}}
\end{equation}
for a.e. $x\in\Omega$, a.e. $\eta=(\eta_1,\eta_2,\ldots,\eta_n)\in\mathbb{R}^n\setminus\{0\}$
and for all $\xi=(\xi_1,\xi_2,\ldots,\xi_n)\in\mathbb{R}^n$,
and

\begin{eqnarray}\label{1.4}
&&|a(x_1,\eta)-a(x_2,\eta)|\\
&\leq&c_2|x_1-x_2|\big[(\kappa+|\eta|^2)^{\frac{p(x_1)-1}{2}}+(\kappa+|\eta|^2)^{\frac{p(x_2)-1}{2}}\big]\big[1+\big|\ln(\kappa+|\eta|^2)^{\frac{1}{2}}\big|\big],\nonumber
\end{eqnarray}
for $x_1,x_2\in\Omega$, $\eta\in\mathbb{R}^n\setminus\{0\}$.

\begin{remark}\label{r1.1}

Assumptions $\eqref{1.2}$, $\eqref{1.3}$ imply \cite{[D]}, \cite{[T]}, for some positive constants $c_3$, $c_4$ and $c_5$
$$
a(x,\xi)\cdot\xi\geqslant c_3(\kappa+|\xi|)^{p(x)}\quad\text{ and }\quad|a(x,\xi)|\leqslant c_4(\kappa+|\xi|)^{p(x)-2}|\xi|.
$$
We therefore include the quasilinear operator
\begin{equation}
Au=\Div\bigg(M(x)\big(\kappa+|\nabla u|^2\big)^{\frac{p(x)-2}{2}}\nabla u\bigg).
\end{equation}
for a bounded Lipschitz positive function or definite positive matrix $M(x)$ uniformly in $x\in\Omega$.
\end{remark}

\begin{remark}\label{r1.2}

The special case $\kappa=0$ corresponds to the heterogeneous $p(x)$-Laplacian operator, which is singular for  $p(x)<2$ and degenerate for  $p(x)>2$. Note that \eqref{1.4} requires  $p(x)$ to be also Lipschitz continuous (see condition $\eqref{2.1}$). In the case of the heterogeneous  $p$-Laplacian, corresponding to the case $p_{-}=p_{+}=p$ in $\eqref{1.1}$, with a Lipschitz coefficient $M(x)$ the assumption $\eqref{1.4}$ is satisfied without the logarithm term and reduces, for all $x_1,x_2\in\Omega$, to
$$
|a(x_1,\eta)-a(x_2,\eta)|\leq c_2|x_1-x_2||\eta|^{p-1}.
$$
\end{remark}

First, we recall the following existence and uniqueness result \cite{[Fu]}, \cite{[RSU]}.
\begin{proposition}\label{p1.1} Assume that $f\in L^{q(\cdot)}(\Omega)$
and $g\in W^{1,p(\cdot)}(\Omega) \cap L^\infty (\Omega)$. Then there
exists a unique solution $u$ to the problem $(P)$.
\end{proposition}

We may prove the following proposition exactly as in Proposition 1.2 of  $\cite{[CL2]}$.

\begin{proposition}\label{p1.2} If $u$ is the solution of
$(P)$ then
\begin{align*}
i)\textrm{ } & f\geqslant0\textit{ in }\Omega~~\Longrightarrow~~0\leqslant u\leqslant\|g\|_{L^\infty}\textit{ in }\Omega.\\
ii)\textrm{ } & Au=f\textit{ in }\mathcal{D}'(\{u>0\}).\\
iii)\textrm{ } & f\chi_{\{u>0\}}\leqslant A u\leqslant f\textit{ a.e. in }\Omega.
\end{align*}
\end{proposition}

\begin{remark}\label{r1.3}
Equation $ii)$ and inequalities $iii)$ of Proposition $\ref{p1.2}$ were
established in \cite{[RSU]}, in the framework of entropy solutions,
under the condition:

\n $\displaystyle{\textrm{ess}\inf_{x\in\Omega}(q_1(x)-(p(x)-1))}>0$,
where $q_1(x)={{q_0(x)p(x)}\over{q_0(x)+1}}$ and  $q_0(x)={{np(x)}\over{n-p(x)}}{{p_--1}\over{p_-}}$.
\end{remark}

\begin{remark}\label{r1.4}
If $f\geqslant 0$ in $\Omega$ or $f\in L^\infty_{\loc}(\Omega)$, we know from Proposition $\ref{p1.2}$
that $u$ is bounded and $Au$ is locally bounded in $\Omega$. Moreover, if $p(x)$ is H\"{o}lder continuous,
and $a(x,\xi)$ satisfies $\eqref{1.2}$-$\eqref{1.4}$, then we have \cite{[F]}, $u\in C_{\loc}^{1,\alpha}(\Omega)$,
for some $\alpha\in(0,1)$.
\end{remark}

In this work we extend classical local properties of the solution and of its free boundary to this more general framework. For $\kappa=0$, in section 2, we establish the growth rate of a class of functions to the heterogeneous case and, in section 3, we obtain the exact growth rate of the solution of the problem $(P)$ near the free boundary, from which we deduce its porosity. These results extend those for the $p$-Laplacian \cite{[KKPS]} and for the $p(x)$-Laplacian \cite{[CL2]}. As a direct consequence, the first inequality of $iii)$ of Proposition $\ref{p1.2}$ is in fact an equation:
$$
Au =f\chi_{\{u>0\}}  \quad \hbox{ a.e.  in }\Omega.
$$
In section 4, also with $\kappa=0$ and constant exponents $1<p<\infty$, we obtain directly the finiteness of the  $\mathcal{H}^{n-1}$-measure of the free boundary  for a larger class of $p$-obstacle type problems that includes degenerate or singular heterogeneous operators, which dependence on $x$ has bounded second order derivatives.
Finally, in the case $\kappa>0$, in section 5, we extend a second order regularity result for the solution of the Dirichlet problem to the class of quasilinear operators following \cite{[CL11]}. This is used in section 6 to obtain, in that case with $\kappa>0$, the local bounded variation of $Au$ for the solution $u$ of the respective obstacle problem, which generalizes the bounded variation estimates of  $\cite{[BK]}$ and yields the control of the  $\mathcal{H}^{n-1}$-measure of the essential free boundary, under the nondegeneracy assumption on $f$.

\section{A class of functions on the unit ball}\label{S2}

In  this section we assume that $\kappa =0$, and in all what follows we assume that $p$ is Lipschitz continuous, that is, there exists a positive constant $L$ such that
\begin{equation}\label{2.1}
 |p(x)-p(y)|\leqslant L|x-y|\qquad\forall x, y\in \Omega.
\end{equation}
We study a family $\mathcal{F}_{a}=\mathcal{F}_{a}(n,c_0,c_1,c_2,p_-,p_+,L)$ of solutions of problems defined on the unit ball
$B_1$. More  precisely, $u\in\mathcal{F}_{a}$ if it satisfies:
$$
\left\{
  \begin{array}{ll}
   u\in W^{1,p(\cdot)}(B_1), \qquad  &u(0)=0, \\
&\\
   0\leqslant u \leqslant 1  \quad \hbox{ in } B_1, \qquad & \|Au \|_{L^\infty (B_1)} \leqslant 1.
  \end{array}
\right.
$$

Condition $u(0)=0$ makes sense, since from \cite{[F]} we know that $u\in C^{1, \alpha}_{\loc}(B_1)$,
for some $\alpha\in (0,1)$. In particular, there exist two positive constants $\alpha=\alpha(n,c_0,c_1,c_2,p_-,p_+,L)$ and $C=C(n,c_0,c_1,c_2,p_-,p_+,L)$ such that

\begin{equation}\label{2.2}
\|u\|_{C^{1,\alpha}(\overline{B}_{3/4})} \leqslant C, \qquad \forall u\in
\mathcal{F}_{a}.
  \end{equation}

\n The following theorem gives a growth rate of the elements in
the class $\mathcal{F}_{a}$.

\begin{theorem}\label{t2.1} There exists a positive  constant
$C_0=C_0(n,c_0,c_1,c_2,p_-,p_+,L)$ such that, for every $u\in \mathcal{F}_{a}$, we
have
$$ 0\leqslant u(x) \leqslant C_0 |x|^{q_0}, \qquad \forall x\in B_1,$$
\n where  $\displaystyle{q_0={{p_0}\over{p_0-1}}}$ is the conjugate of $p_0=p(0)$.
\end{theorem}

Let us first introduce some notations. For a nonnegative
bounded function $u$, we define the quantity
$\displaystyle{S(r,u) = \sup_{x\in B_r} u(x)}.$ We also define, for each
$u\in\mathcal{F}_a$, the set

$$
\mathbb{M}(u) = \{ j\in \mathbb{N} :\quad 2^{q_0}S( 2^{-j-1},u)\geqslant S( 2^{-j},u) \}.
$$

Then we have

\begin{lemma}\label{l2.1} If $ \mathbb{M}(u)
\neq\emptyset$, then there exists  a constant
$\tilde{c}_0$ depending only on $n$, $c_0$, $c_1$, $c_2$, $p_-$, $p_+$ and $L$ such that
$$
S( 2^{-j-1},u) \leqslant \tilde{c}_0( 2^{-j})^{q_0}, \qquad
\forall u\in \mathcal{F}_a , \quad \forall j\in
\mathbb{M}(u).
$$
\end{lemma}

\begin{proof} Arguing by contradiction, we assume
that $\forall k\in\mathbb{N}$ there exists $u_k\in\mathcal{F}_a$ and $j_k\in\mathbb{M}(u_k)$ such that
\begin{equation}\label{2.3}
  S( 2^{-j_k-1},u_k) \geqslant k ( 2^{-j_k})^{q_0}.
  \end{equation}
Consider the function
$$
v_k(x)= \displaystyle{ u_k( 2^{-j_k} x) \over
{S(2^{-j_k-1},u_k)}}
$$
defined  in $B_1$.  By definition of $v_k$
and $\mathbb{M}(u_k)$, we have

$$\left\{
  \begin{array}{ll}
    &0\leqslant v_k\leqslant\displaystyle{ {{S(2^{-j_k},u_k)} \over
{S(2^{-j_k-1},u_k)}}}\leqslant 2^{q_0}\quad \hbox{ in } B_1,\\
&\\
    &\displaystyle{\sup_{x\in \overline{B}_{1/2}} v_k(x)}=1,\qquad v_k(0)=0.
  \end{array}
\right.
$$
Now, let $p_k(x)=p(2^{-j_k}x)$, $s_k=\displaystyle{ 2^{-j_k} \over
{S(2^{-j_k-1},u_k)}}$, and define for $(x,\xi)\in B_1\times\mathbb{R}^n$
\begin{equation}\label{2.4.0}
a^k(x,\xi):=s_k^{p_k(x)-1}a(2^{-j_k}x,\frac{1}{s_k}\xi).
\end{equation}
We claim that
\begin{equation}\label{2.4}
|A_k v_k(x)|:=|\textrm{div}(a^k(x,\nabla v_k(x)))|\rightarrow0\quad\textrm{ as }\quad k\rightarrow\infty.
\end{equation}
Then one can easily verify that
\begin{align*}
A_k v_k(x)& =2^{-j_k}s_k^{p_k(x)-1}(Au_k)(2^{-j_k}x)\\
&\quad+ 2^{-j_k}(\ln(s_k))s_k^{p_k(x)-1}a(2^{-j_k}x,\nabla u_k(2^{-j_k}x))\nabla
p(2^{-j_k}x).
  \end{align*}
Using the structural assumptions (second inequality in Remark $\ref{r1.1}$) and the fact that $u_k \in \mathcal{F}_a$, and $| \nabla
p |_{L^\infty(\Omega)}\leqslant L$ (by $\eqref{2.1}$), this leads to
\begin{equation*}
|A_k v_k(x)|\leqslant 2^{-j_k}s_k^{p_k(x)-1}+
c_4L2^{-j_k}|\ln(s_k)|s_k^{p_k(x)-1}|\nabla
u_k(2^{-j_k}x)|^{p_k(x)-1}.
\end{equation*}
Since $u_k\geqslant 0$ in $B_1$, $u_k(0)=0$, and $u_k\in
C^1(\overline{B}_{3/4})$, we have $\nabla u_k(0)=0$. Combining this
result and $\eqref{2.2}$, we get

\begin{equation*}
\forall k\in  \mathbb{N},\quad \forall x\in B_{1}\quad |\nabla
u_k(2^{-j_k}x)|\leqslant C(2^{-j_k})^\alpha.
  \end{equation*}

\n It follows  that

\begin{equation}\label{2.5}
|A_k v_k(x)|\leqslant 2^{-j_k}s_k^{p_k(x)-1}(1+
c_4L(C)^{p_k(x)-1}|\ln(s_k)|(2^{-j_k})^{\alpha(p_k(x)-1)}).
\end{equation}

\n Note that $S(2^{-j_k-1},u_k)=u_k(z_k)$, for some
$z_k\in \overline{B}_{2^{-j_k-1}}$. Since $u_k(0)=0$ and $u_k\in
C^1(\overline{B}_{3/4})$, we deduce that

$$ S(2^{-j_k-1},u_k)\leqslant C|z_k|\leqslant C2^{-j_k-1}.$$

\n Consequently, we obtain
\begin{equation*}
s_k=\displaystyle{{ 2^{-j_k} \over {S(2^{-j_k-1},u_k)}}\geqslant
{2^{-j_k} \over {C2^{-j_k-1}}}={2\over C}}=\mu.
\end{equation*}
We recall from $\cite{[CL2]}$ that there exist positive constants $\tilde{c}_1=\tilde{c}_1(\alpha,p_0,\mu)$ and $\tilde{c}_2=\tilde{c}_2(\alpha,L,p_0,\mu)$ such that
$$
|\ln(s_k)|(2^{-j_k})^{\alpha(p_k(x)-1)}\leqslant\frac{\tilde{c}_1}{k^{\alpha(p_0-1)^2}}\quad\textrm{and}\quad 2^{-j_k}s_k^{p_k(x)-1}\leqslant\frac{\tilde{c}_2}{k^{p_0-1}},\quad\forall k\in\mathbb{N},
$$
which together with $\eqref{2.5}$ gives $\eqref{2.4}$.

\begin{lemma}\label{l2.2.0}
      With the notation above, the mapping $a^k(x,\xi)$ defined in $\eqref{2.4.0}$ satisfies all structural conditions (with the same constants as $a(x,\xi)$). Moreover, we have uniformly in $(x,\xi)\in B_1\times B_M$, for any $M>0$
\begin{equation}\label{2.10}
\bigg|\frac{\partial a^k_i}{\partial x_j}\bigg|\leqslant L_k\rightarrow0\quad\textrm{ as }\quad k\rightarrow\infty.
\end{equation}
\end{lemma}
\begin{proof}
It is easy to see that
\begin{eqnarray*}
\sum_{i,j=1}^n\frac{\partial a^k_i}{\partial\eta_j}(x,\eta)\xi_i\xi_j&=&\sum_{i,j=1}^n s_k^{p_k(x)-1}\frac{1}{s_k}\frac{\partial a_i}{\partial\eta_j}(2^{-j_k}x,\frac{1}{s_k}\eta)\xi_i\xi_j\nonumber\\
&\geqslant& c_0s_k^{p_k(x)-2}\bigg|\frac{\eta}{s_k}\bigg|^{p_k(x)-2}|\xi|^2\nonumber\\
&=&c_0|\eta|^{p_k(x)-2}|\xi|^2.
\end{eqnarray*}
\begin{eqnarray*}
\sum_{i,j=1}^n\bigg|\frac{\partial a^k_i}{\partial\eta_j}(x,\eta)\bigg|&=&\sum_{i,j=1}^ns_k^{p_k(x)-1}\frac{1}{s_k}\bigg|\frac{\partial a_i}{\partial\eta_j}(2^{-j_k}x,\frac{1}{s_k}\eta)\bigg|\nonumber\\
&\leqslant& c_1s_k^{p_k(x)-2}\bigg|\frac{\eta}{s_k}\bigg|^{p_k(x)-2}\nonumber\\
&=&c_1|\eta|^{p_k(x)-2}.
\end{eqnarray*}
Now, to prove $\eqref{2.10}$, we use the second inequality in Remark $\ref{r1.1}$ and $\eqref{1.4}$
\begin{eqnarray*}
\bigg|\frac{\partial a^k_i}{\partial x_j}\bigg|&=&\bigg|\frac{\partial}{\partial x_j}\bigg(s_k^{p_k(x)-1}a_i\big(2^{-j_k}x,\frac{1}{s_k}\xi\big)\bigg)\bigg|\nonumber\\
&\leqslant&\big|\nabla\big(s_k^{p_k(x)-1}\big)\big|\big|a_i\big(2^{-j_k}x,\frac{1}{s_k}\xi\big)\big|\nonumber\\
&+&2^{-j_k}s_k^{p_k(x)-1}\bigg|\frac{\partial a_i}{\partial x_j}\big(2^{-j_k}x,\frac{1}{s_k}\xi\big)\bigg|\nonumber\\
&\leqslant& c_4L2^{-j_k}s_k^{p_k(x)-1}|\ln(s_k)|\bigg|\frac{\xi}{s_k}\bigg|^{p_k(x)-1}\nonumber\\
&+&2c_22^{-j_k}s_k^{p_k(x)-1}\bigg|\frac{\xi}{s_k}\bigg|^{p_k(x)-1}\bigg|\ln\big|\frac{\xi}{s_k}\big|\bigg|\nonumber\\
&=&\bigg(c_4L2^{-j_k}|\ln(s_k)|+2c_22^{-j_k}\big|\ln|\frac{\xi}{s_k}|\big|\bigg)|\xi|^{p_k(x)-1}=:L_k
\end{eqnarray*}
On the other hand,
\begin{eqnarray*}
2^{-j_k}|\xi|^{p_k(x)-1}\big|\ln|\frac{\xi}{s_k}|\big|&=&2^{-j_k}|\xi|^{p_k(x)-1}|\ln(|\xi|)-\ln(s_k)|\nonumber\\
&\leqslant&2^{-j_k}|\xi|^{p_k(x)-1}|\ln(|\xi|)|\nonumber\\
&+&2^{-j_k}|\ln(s_k)||\xi|^{p_k(x)-1}
\end{eqnarray*}
The first term uniformly goes to zero (for $(x,\xi)\in B_1\times B_M$, for any $M>0$) when $k\rightarrow\infty$.
Since $2^{-j_k}|\ln(s_k)|\rightarrow0$ as $k\rightarrow0$ ($\cite{[CL2]}$), so does the second term.
\end{proof}
Therefore, the pointwise limit of $a^k(x,\xi)$ does not depend on $x$:
$$
a^k(x,\xi)\rightarrow\tilde{a}(\xi),
$$
where $\tilde{a}$ is a vector field satisfying the same structural assumptions $\eqref{1.2}$, $\eqref{1.3}$, with $p(x)$ replaced by $p_0=p(0)$.\\

${\emph{Conclusion of the proof of Lemma $\ref{l2.1}$}}$. By taking into account the uniform bound of $v_k$,
$\eqref{2.4}$, and the fact that $p_k$ satisfies $\eqref{1.1}$ and $\eqref{2.1}$ with the
same constants, we deduce \cite{[F]} that there exist two
positive constants $\delta$ and $C$, independent of $k$,
such that $v_k\in C^{1,\delta}(\overline{B}_{3/4})$ and
$\|v_k\|_{C^{1,\delta}(\overline{B}_{3/4})} \leqslant C$, for all $k\geqslant k_0$.
It follows then from the Ascoli-Arzella's theorem that there exists
a subsequence, still denoted by $v_k$, and a function $v\in
C^{1,\delta'}(\overline{B}_{3/4})$ such that $v_k \longrightarrow
v$  in $ C^{1,\delta'}(\overline{B}_{3/4})$, for any $\delta'\in(0,\delta)$. Moreover, it is clear that
 $v$ satisfies (in the weak sense)

$$
\left\{
    \begin{array}{ll}
 & \textrm{div}\big(\tilde{a}(\nabla v)\big) =0 \quad \hbox{ in } B_{3/4},\qquad  v\geqslant
0\quad \hbox{ in } B_{3/4},\\
&\\
&\displaystyle{\sup_{x\in B_{1/2}} v(x)}=1,\qquad v(0)=0.
    \end{array}
  \right.
$$

\n By the strong maximum principle (see $\cite{[HKM]}$, for instance) we have necessarily $v\equiv 0$ in
$B_{3/4}$, which is in contradiction with
$\displaystyle{\sup_{x\in B_{1/2}} v(x)=1}$.

\end{proof}

\vs 0,5cm \n \emph{Proof of  Theorem 2.1.} The theorem is proved by induction. Using Lemma
$\ref{l2.1}$, the proof follows step by step as the one of Theorem 2.1 of $\cite{[CL2]}$
\qed

\section{Porosity of the free boundary for $\kappa=0$ }\label{S3}

In this section we also assume $\kappa=0$ and that there exist positive
constants $\lambda$,  $\Lambda$, such that,

\begin{equation}\label{3.1}
0< \lambda \leqslant f \leqslant \Lambda<\infty, \quad \hbox{a.e. in
}\Omega.
\end{equation}
\vs 0,2cm \n The following lemma and Theorem $\ref{t2.1}$ give the exact growth rate of
the solution of the problem $(P)$ near the free boundary. This extends to the heterogeneous
$a(x,\eta)$-case with $\kappa=0$ the results
established in \cite{[C]} for the Laplacian and generalized in \cite{[KKPS]}
for the $p$-Laplacian, as well as for the $A$-Laplacian in \cite{[CL1]} and
for the homogeneous $p(x)$-Laplacian in  \cite{[CL2]}.

\begin{lemma}\label{l3.1} Suppose that $u\in  W^{1,p(\cdot)}(\Omega)$
is a nonnegative continuous function satisfying
$$
Au=f  \quad \hbox{ in
}\quad{\cal D}'(\{u>0\}).
$$
\n Then there exists $r_*>0$ such that
for each  $y\in \overline{\{u>0\}}$ and $r\in(0,r_*)$ satisfying
$B_r(y)\subset \Omega$, we have for an appropriate constant $C(y)>0$
$$
\sup_{\partial B_r(y)}u \geqslant
C(y) r^{{p(y)}\over{p(y)-1}}+u(y).
$$

\end{lemma}

\begin{proof} It is enough to prove the result for $y\in
\{u>0\}$. For each $y$, we consider the function defined by
$$
v(x):=v(x,y):=C(y)|x-y|^{{p(y)}\over{p(y)-1}},
$$
where $C(y)$ is to be chosen later.\\
\n We claim that there exists $r_*>0$ such that
\begin{equation}\label{3.2}
\forall r\in(0,r_*),\quad  \forall y\in \Omega,\quad \forall x\in
B_r(y)\subset\Omega\qquad Av\leqslant\lambda.
\end{equation}

\n To prove $\eqref{3.2}$, we compute $\nabla_x v$ and the divergence of $a(x,\nabla_xv)$:
\begin{eqnarray*}
\Div\big(a(x,\nabla v)\big)&=&\Div\big(a(x,C(y)q(y)|x-y|^{q(y)-2}(x-y)\big)\nonumber\\
&=&\sum_{i=1}^{n}\frac{\partial a_i}{\partial x_i}(x,w)+\sum_{i,j=1}^n\frac{\partial a_i}{\partial\eta_j}(x,w)\cdot\frac{\partial w_j}{\partial x_i}(x)\nonumber\\
&=&\sum_{i=1}^{n}\frac{\partial a_i}{\partial x_i}+C(y)q(y)|x-y|^{q(y)-2}\sum_{i,j=1}^n\bigg(\delta_{ij}\nonumber\\
&+&(q(y)-2)\frac{(x_i-y_i)(x_j-y_j)}{|x-y|^2}\bigg)\frac{\partial a_i}{\partial\eta_j},
\end{eqnarray*}
where $w(x):=C(y)q(y)|x-y|^{q(y)-2}(x-y)$.\\
Therefore, using the structural assumptions $\eqref{1.3}$, $\eqref{1.4}$, we get
\begin{eqnarray*}
&&|\textrm{div}\big(a(x,\nabla v)\big)|\leqslant2c_2|w|^{p(x)-1}\big|\ln|w|\big|\nonumber\\
&&~~+c_1\max(1,q(y)-1)\big(C(y)q(y)\big)^{p(x)-1}|x-y|^{(q(y)-1)(p(x)-2)+q(y)-2}\nonumber\\
&&=:S_1+S_2.
\end{eqnarray*}
To estimate $S_1$, we write
\begin{eqnarray*}
S_1&=&2c_2|w|^{p(x)-1}|\ln(|w|)|\nonumber\\
&=&2c_2\big(C(y)q(y)\big)^{p(x)-1}|x-y|^{(p(x)-1)(q(y)-1)}\big|\ln\big(C(y)q(y)\big)+(q(y)-1)\ln|x-y|\big|\nonumber\\
&\leqslant&2c_2\big(q(y)\big)^{p(x)-1}\big(C(y)\big)^{p(x)-1}|x-y|^{(p(x)-1)(q(y)-1)}\big|\ln\big(C(y)q(y)\big)\big|\nonumber\\
&&+2c_2(q(y)-1)\big(C(y)q(y)\big)^{p(x)-1}|x-y|^{(p(x)-1)(q(y)-1)}\big|\ln(|x-y|)\big|
\end{eqnarray*}
Since $r\ln r\rightarrow0$, when $r\rightarrow0$, then $S_1$ can be made as small as we wish, if $x$ is close to $y$, and $C(y)$ is small enough.
To estimate $S_2$, we first observe that
$$
|x-y|^{(q(y)-1)(p(x)-2)+q(y)-2}=|x-y|^{\frac{p(x)-p(y)}{p(y)-1}}
$$
and for $|x-y|<r<{1\over e}$, we have
\begin{equation*}
|x-y|^{{p(x)-p(y)}\over{p(y)-1}}=e^{{{p(x)-p(y)}\over{p(y)-1}}\ln(|x-y|)}
\leqslant e^{{{L}\over{p_--1}}|x-y||\ln(|x-y|)|} \leqslant
e^{{{L}\over{p_--1}}r|\ln(r)|},
\end{equation*}
and since
\begin{eqnarray*}
S_2&=&c_1\max(1,q(y)-1)\big(C(y)q(y)\big)^{p(x)-1}|x-y|^{\frac{p(x)-p(y)}{p(y)-1}}\nonumber\\
&\leqslant&c_1\max(1,q(y)-1)\big(C(y)q(y)\big)^{p(x)-1}e^{{{L}\over{p_--1}}r|\ln(r)|},
\end{eqnarray*}
$S_2$ also can be made small, if $r$ and $C(y)$ are small enough.

\n It is clear now that $\eqref{3.2}$ holds.

\vs 0.5cm\n  Now let $\epsilon
>0$ and consider the following function $u_\epsilon(x) = u(x)-(1-\epsilon)u(y).$

\n We have from $\eqref{3.1}$-$\eqref{3.2}$

$$
Au_\epsilon=Au=f\geqslant
\lambda\geqslant Av\quad \hbox{ in }\quad B_r(y)\cap
\{u>0\}.
$$
Moreover,
$$
u_\epsilon=-(1-\epsilon)u(y)\leqslant 0\leqslant
v\quad \hbox{ on }\quad(\partial\{u>0\})\cap B_r(y).
$$
If we also have
$$
u_\epsilon\leqslant v\quad \hbox{ on
}\quad(\partial B_r(y))\cap \{u>0\},
$$
then we get by the weak maximum principle
$$
u_\epsilon\leqslant v\quad \hbox{ in
}\quad B_r(y)\cap \{u>0\}.
$$
But $u_\epsilon(y)=\epsilon u(y)>0=v(y)$, which constitutes a
contradiction.

\vs 0,2cm\n So there exists $z\in (\partial B_r(y))\cap\{u>0\}$
such that $u_\epsilon(z)>v(z)$. Since $v$ is radial, we get
\begin{align*}
\sup_{\partial B_r(y)}(u -(1-\epsilon)u(y))=\sup_{\partial
B_r(y)}u_\epsilon&\geqslant \sup_{\partial B_r(y)\cap
\{u>0\}}u_\epsilon\geqslant  u_\epsilon(z)\\
&> v(z)=C(y)r^{{p(y)}\over{p(y)-1}}.
\end{align*}

\n Letting $\epsilon\rightarrow 0$, we get
$$
\displaystyle{\sup_{\overline{B}_r(y)}u\geqslant\sup_{\partial
B_r(y)}u\geqslant C(y) r^{{p(y)}\over{p(y)-1}}+u(y).}
$$
\end{proof}

Denoting by $u$ the solution of the problem $(P)$ of the Introduction, we may now prove the main result
of this section: the porosity of the free boundary
$\partial\{u>0\}\cap\Omega$.

\n We recall that a set $E\subset \mathbb{R}^n$ is called
porous  with porosity $\delta$, if there is an $r_0>0$ such that
$$
\forall x\in E , \quad \forall r\in (0,r_0),\quad \exists y\in
\mathbb{R}^n\quad \hbox{ such that } \quad B_{\delta r}(y)\subset
B_{r}(x)\setminus E.
$$

\n A porous set of porosity $\delta$ has Hausdorff dimension not exceeding
$n-c\delta^n$, where $c=c(n) >0$ is a constant depending only
on $n$. In particular, a porous set has Lebesgue measure zero
(see $\cite{[MV]}$ or $\cite{[Z]}$ for instance).

\begin{theorem}\label{t3.1} Let $r_*$ be as in Lemma $\ref{l3.1}$,
$R\in(0, r_*)$ and $x_0\in\Omega$ such that
$\overline{B_{4R}(x_0)}\subset\Omega$. Then $\partial
\{u>0\}\cap \overline{B_R(x_0)} $ is porous with porosity constant
depending only on  $n, p_-, p_+, L,$ $c_0, c_1, c_2$, $\lambda$,
$\Lambda$, $R$, and $\| g\|_{L^\infty}$. As an immediate consequence, we have
$$
Au =f\chi_{\{u>0\}}
  \quad \hbox{ a.e.  in }\Omega.
$$
\end{theorem}

We need first a lemma.

\begin{lemma}\label{l3.2} Let $R>0$ and $x_0\in\Omega$ such that
$\overline{B_{4R}(x_0)}\subset\Omega$. We consider,
for $y_0\in \overline{B_{2R}(x_0)}\cap\{u=0\}$ and $M>0$, the functions defined in $\overline{B}_1$ by
\begin{equation}\label{3.3}
\bar{a}(z,\xi)=a(y_0+Rz,M\xi),\qquad
\bar{u}(z)={u(y_0+Rz)\over MR}.
\end{equation}
Then we have $\bar{u}\in {\cal F}_{\bar{a}}$, for all $R\leqslant R_0=\frac{1}{\Lambda}$ and
$M\geqslant M_0={{\| g\|_{L^\infty}}\over R}$, where $ {\cal F}_{\bar{a}}$ is defined as in Section 2 with the operator corresponding to $\bar{a}$.
\end{lemma}

\begin{proof} First, note that $\bar{a}$ and $\bar{u}$ are well defined, since
we have $\overline{B_R(y_0)}\subset \overline{B_{3R}(x_0)}\subset\Omega$.
Moreover, we have $\displaystyle{\bar{u}(0)={u(y_0)\over MR}=0}$, and for
$\displaystyle{M\geqslant {{\| g\|_{L^\infty}}\over R} }$, we have $0\leqslant\bar{u}\leqslant1$ in $B_1$.\\
Note that $\bar{a}(z,\xi)$ satisfies all structural conditions (not necessarily with the same constants as
for $a$) with $\bar{p}(z):=p(y_0+Rz)$ instead of $p$.

\n Next, one can easily verify that $\bar{u}$ satisfies
\begin{eqnarray*}
\bar{A}\bar{u}&:=&\textrm{div}\big(\bar{a}(z,\nabla\bar{u}(z))\big)\nonumber\\
&=&\textrm{div}\big(a(y_0+Rz,\nabla u(y_0+Rz))\big)\nonumber\\
&=&R(Au)(y_0+Rz)\leqslant R\Lambda\leqslant1
\end{eqnarray*}
if $R\leqslant R_0=\frac{1}{\Lambda}$, and we conclude that $\bar{u}\in {\cal F}_{\bar{a}}$ for all $M\geqslant M_0$
and $R\leqslant R_0$.
\end{proof}

\emph{Proof \,of \,Theorem $\ref{t3.1}$.} Now, to prove the theorem, we argue as in $\cite{[CL2]}$.
Let $r_*$ be as in Lemma $\ref{l3.1}$ and $R_*=\min(r_*,R_0)$. Let then $R\in (0,R_*)$ be such that
$\overline{B_{4R}(x_0)}\subset \Omega$,
and let $x\in E=\partial\{u>0\}\cap \overline{B_R(x_0)}$.
For each $0<r<R$, we have $\overline{B_r(x)}\subset B_{2R}(x_0)\subset \Omega$.
Let $y\in \partial B_r(x)$ such that $u(y)=\displaystyle{\sup_{\partial B_r(x)} u}$.
Then we have by Lemma $\ref{l3.1}$

\begin{equation}\label{3.4}
u(y)\geqslant C_0'r^{{p(x)}\over{p(x)-1}}+u(x)=C_0'r^{{p(x)}\over{p(x)-1}}.
\end{equation}

\n Hence $y\in B_{2R}(x_0)\cap \{u>0\}$. Denoting by
$d(y)=dist(y,\overline{B_{2R}(x_0)}\cap \{u=0\})$ the distance from
$y$ to the set $\overline{B_{2R}(x_0)}\cap \{u=0\}$, we get
from Lemma $\ref{l2.1}$ and Lemma $\ref{l3.2}$, for a constant $C_0$
\begin{equation}\label{3.5}
u(y)\leqslant C_0(d(y))^{{p(y_0)}\over{{p(y_0)-1}}}.
\end{equation}

\n Then we deduce from $\eqref{3.4}$-$\eqref{3.5}$ that
\begin{equation}\label{3.6}
C_0' r^{{p(x)}\over{p(x)-1}}\leqslant u(y)\leqslant
C_0(d(y))^{{p(y_0)}\over{{p(y_0)-1}}},
\end{equation}
which, by using the Lipschitz continuity of $p(x)$, leads to (see the proof of Theorem 3.1 in $\cite{[CL2]}$)
$$
d(y) \geqslant\delta r,
$$
where $\delta>0$ is some constant smaller than one and depending only on  $n, p_-, p_+, L,$ $c_0, c_1, c_2$, $\lambda$,
$\Lambda$, $R$, and $\| g\|_{L^\infty}$.

Let now  $y^*\in [x,y]$ such that $|y-y^*|=\delta r/2$. Then we
have $\cite{[CL2]}$
$$
B_{{\delta\over 2}r}(y^*)\subset B_{\delta r}(y)\cap B_r(x).
$$
Moreover,  we have
$$
B_{{\delta }r}(y)\cap  B_r(x)\subset \{u>0\},$$

\n since $ B_{{\delta}r}(y)\subset B_{d(y)}(y)\subset\{u>0\}$ and $d(y) \geqslant \delta r$.

\n Hence we obtain
$$
B_{{\delta\over 2}r}(y^*)\subset B_{\delta r}(y)\cap B_r(x)\subset B_r(x)\setminus \partial \{u>0\}
\subset B_r(x)\setminus E.
$$ \qed

Note that as a consequence of Theorem $\ref{t2.1}$ and Lemma $\ref{l3.2}$, we may also obtain a more explicit
 growth rate of the solution $u$ of the problem $(P)$ near the free boundary.

\begin{proposition}\label{p3.1} Let $R_0>0$ be as in Lemma 3.2, $R\in(0,R_0)$ and
$x_0\in\Omega$ such that $u(x_0)=0$ and $\overline{B_{4R}(x_0)}\subset\Omega$.
Then there exists a positive constant $\widetilde{C}_0$ depending only on $n$, $p_-, p_+, L,$
$\Lambda$, $c_0$, $c_1$, $c_2$, and $\| g\|_{L^\infty}$ such that we have
$$
u(x)\leqslant \widetilde{C}_0 |x-x_0|^{{p(x_0)}\over{p(x_0)-1}} \qquad \forall x\in B_R(x_0).
$$
\end{proposition}

\begin{proof}
Let $R$ and $x_0$ be as in the proposition. Consider the functions $\bar{a}(y,\xi)$ and
$\bar{u}(y)$ defined in Lemma $\ref{l3.2}$, for $M>0$. By Lemma $\ref{l3.2}$, there exists
$M_0$ such that for all $M\geqslant M_0$ we have $\bar{u}\in{\cal F}_{\bar{a}}$. Applying Theorem
$\ref{t2.1}$ for $M= M_0$ and $R=R_0$, we obtain for a positive constant $C_0>0$ depending
only on $n$, $p_-, p_+, L,$ $c_0$, $c_1$, $c_2$
$$
\bar{u}(y)\leqslant C_0|y|^{{\bar{p}(0)}\over \bar{p}(0)-1} \qquad \forall y\in B_1.
$$

\n Taking $y={{|x-x_0|}\over R_0}$ for $x\in B_{R}(x_0)$, we get
$$
u(x)\leqslant{{C_0M_0R_0}\over {R_0^{{p(x_0)}\over{p(x_0)-1}} }} |x-x_0|^{{p(x_0)}\over {p(x_0)-1}}
={{C_0\| g\|_{L^\infty}}\over {R_0^{{p(x_0)}\over{p(x_0)-1}} }} |x-x_0|^{{p(x_0)}\over {p(x_0)-1}}
=\widetilde{C}_0 |x-x_0|^{{p(x_0)}\over{p(x_0)-1}}.
$$
\end{proof}

\section{The Obstacle Problem of $p$-Laplacian Type in a  Heterogeneous Case}

In this section we consider still the case of $\kappa=0$ and we assume the exponent $p$ is a constant, $1<p<\infty$.
For simplicity, since the results are local, we restrict ourselves to the unit ball, and assume that
\begin{eqnarray}\label{e4.1}
0<f \leq\Lambda<\infty\quad \text{a.e. in }~B_1,
\end{eqnarray}
and additionally, $\nabla f\in\mathcal{M}_{\loc}^n(B_1)$,
which means that there exists a positive constant $C_0$ such that
\begin{equation}\label{e4.2}
\int_{B_{r}}|\nabla f|\,dx\leq C_0r^{n-1},\quad\forall r\in(0,3/4).
\end{equation}
In particular $\eqref{e4.2}$ is satisfied, if $f\in C^{0,1}(\overline{B}_1)$.

We assume that $a$ satisfies $\eqref{1.2}$ for $\kappa=0$, and satisfies for two positive constants
$c_3$ and $c_4$, for a.e. $(x,\eta)\in\Omega\times\mathbb{R}^n$.
\begin{equation}\label{e4.3}
\sum_{i,j=1}^n\bigg|\frac{\partial^2 a_i}{\partial x_i\partial x_j}(x,\eta)\bigg|
\leq c_3|\eta|^{p-1},
\end{equation}
\begin{equation}\label{e4.4}
\sum_{i,j,k=1}^n\bigg|\frac{\partial^2 a_k}{\partial \eta_j\partial x_i}(x,\eta)\bigg|
\leq c_4|\eta|^{p-2}.
\end{equation}

Note that $\eqref{e4.4}$ implies $\eqref{1.3}$ and that $\eqref{e4.3}$ implies that $a$ satisfies
\begin{equation}\label{e4.5}
\sum_{i,k=1}^n\bigg|\frac{\partial a_k}{\partial x_i}(x,\eta)\bigg|\leqslant c_2|\eta|^{p-1}
\end{equation}
which is the equivalent of $\eqref{1.4}$, when $p$ is constant, as in Remark 1.2.

\subsection{Some auxiliary lemmas for a class of functions on the unit ball}

We consider the solutions of the following class of problems
\begin{equation*}\mathcal{F}_{a(\cdot)}:
\begin{cases}
&u\in W^{1,p}(B_1)\cap C^{1,\alpha}(B_1),\\
&\Div\big(a(x,\nabla u(x))\big)=f(x)\hbox{ in }\{u>0\}\cap B_1,\\
&0\leq u\leq M_0\hbox{ in }B_1,\\
&0\in\partial\{u>0\},
\end{cases}
\end{equation*}
where $M_0$ is a positive constant.

We introduce for each $\epsilon\in(0,1)$, the unique
solution of the following approximating problem
\begin{equation}\label{e4.6}
\begin{cases}
&u_\epsilon-u\in W^{1,p}_0(B_1),\\
&\Div\big(a_\epsilon(x,\nabla u_\epsilon)\big)=fH_\epsilon(u_\epsilon)\quad\hbox{ in }B_1,
\end{cases}
\end{equation}
where $H_\epsilon$ is an approximation of the Heaviside function defined by $H_\epsilon(v):=\min(1,\frac{v^+}{\epsilon})$, and $a_\epsilon$ is given by:
$$
a_\epsilon(x,\eta):=a(x,\eta)+\frac{\epsilon c_0}{n}\big(\epsilon+|\eta|^2\big)^{\frac{p-2}{2}}\eta,\,\,x\in\Omega,\,\eta\in\mathbb{R}^n.
$$
Note that $a_\epsilon$ satisfies $\eqref{1.2}$-$\eqref{1.3}$ for $\kappa=\epsilon$,
because $a$ satisfies the same inequalities for $\kappa=0$.
Moreover taking into account \eqref{e4.3}-$\eqref{e4.4}$, we can easily
verify that we have for a.e. $(x,\eta)\in\Omega\times\mathbb{R}^n$
\begin{equation}\label{e4.7}
\sum_{i,k=1}^n\bigg|\frac{\partial a_{\epsilon k}}{\partial x_i}(x,\eta)\bigg|\leq c_2(\epsilon+|\eta|^2)^{{p-1}\over2},
\end{equation}
\begin{equation}\label{e4.8}
\sum_{i,j=1}^n\bigg|\frac{\partial^2 a_{\epsilon i}}{\partial x_i\partial x_j}(x,\eta)\bigg|
\leq c_3(\epsilon+|\eta|^2)^{{p-1}\over2},
\end{equation}
\begin{equation}\label{e4.9}
\sum_{i,j,k=1}^n\bigg|\frac{\partial^2 a_{\epsilon k}}{\partial \eta_j\partial x_i}(x,\eta)\bigg|
\leq c_4(\epsilon+|\eta|^2)^{{p-2}\over2}.
\end{equation}

First, we observe \cite{[D]}, \cite{[T]} that there exist two constants $\alpha\in(0,1)$
and $M_1>1$ depending only on $n$, $p$, $c_0$, $c_1$, $c_2$, $\Lambda$, and $M_0$ such that
$u_\epsilon\in C^{1,\alpha}_{\loc}(B_1)$ and
\begin{equation}\label{e4.10}
\|u_\epsilon\|_{C^{1,\alpha}(\overline{B}_{3/4})}\leqslant M_1.
\end{equation}
In particular, if we set $t_\epsilon=(\epsilon+|\nabla u_\epsilon|^2)^{1/2}$,
then we can assume without loss of generality, that
\begin{equation}\label{e4.11}
\|t_\epsilon\|_{L^\infty(B_{3/4})}\leqslant M_1.
\end{equation}
Adapting part of the proof of Proposition 2.1 in \cite{[CLR]}, we see
that there exists a subsequence, still denoted by $u_\epsilon$ such that
\begin{equation}\label{e4.12}
u_\epsilon\rightarrow u \quad\text{in } C_{\loc}^{1,\beta}(B_1)\quad\text{ for all }\beta\in(0,\alpha).
\end{equation}
Moreover, we know from Theorem $\ref{t4.1}$ that
\begin{equation}\label{e4.13}
u_\epsilon\in W^{2,2}(B_{3/4}).
\end{equation}
For each $r\in(0,1/2)$ and $\epsilon\in(0,1)$, we introduce
the following quantity
\begin{eqnarray*}
E_\epsilon(r,v)&=&\frac{1}{|B_r|}\int_{B_r}\big[(\epsilon+|\nabla v|^2)^{{p-2}\over 2}|D^2v|\big]^2\,dx.
\end{eqnarray*}
The first result is an estimate of $E_\epsilon(1/2,u_\epsilon)$.
\begin{lemma}\label{l4.1}
Assume that $p$ is constant, $f$ satisfies $\eqref{e4.1}$-$\eqref{e4.2}$, and that $a$ satisfies
$\eqref{1.2}$-$\eqref{1.3}$ for $\kappa=0$, and $\eqref{e4.3}$-$\eqref{e4.4}$. Then we have for any $\epsilon\in(0,1)$
\begin{eqnarray}\label{e4.14}
&&E_\epsilon(1/2,u_\epsilon)\leqslant  {{3^n(4c_1'\sqrt{n}+c_4)^2+2c_3c_0'}\over{2^nc_0'^2\min(1,p-1)^2}}|B_{3/4}|\|t_\epsilon\|_{L^\infty(B_{3/4})}^{2(p-1)}\nonumber\\
&&~~+{{2\sqrt{n}}\over{c_0'\min(1,p-1)|B_{1/2}|}}\|t_\epsilon\|_{L^\infty(B_{3/4})}^{p-1}\int_{B_{3/4}}|\nabla f| dx.
\end{eqnarray}
\end{lemma}
To prove Lemma $\ref{l4.1}$, we need the following lemma:

\begin{lemma}\label{l4.2} Let $G$ be a smooth odd nondecreasing function, and $\zeta$ a nonnegative smooth function
with compact support in $B_1$. Then we have
\begin{eqnarray}\label{e4.15}
&&c_0'\int_{B_1}\zeta^2 \sum_{i} G'(u_{\epsilon x_i})t_\epsilon^{p-2}|\nabla u_{\epsilon x_i}|^2 dx\nonumber\\
&&\leq \sqrt{n}c_1'\int_{B_1}\zeta G(t_{\epsilon})t_{\epsilon}^{p-2}|D^2 u_\epsilon||\nabla \zeta| dx\nonumber\\
&&~~+c_3\int_{B_1}\zeta^2 G(t_{\epsilon})t_\epsilon^{p-1}dx+c_4\int_{B_1} \zeta^2 G(t_{\epsilon}) t_\epsilon^{p-2}
|D^2 u_\epsilon|dx\nonumber\\
&&~~+\sqrt{n}\int_{B_1}\zeta^2 G(t_{\epsilon})|\nabla f|dx.
\end{eqnarray}
\end{lemma}

\vs 0.3cm\n \emph{Proof.} Let $G$ and $\zeta$ be as in the lemma.
Note that \cite{[T]}
\begin{equation}\label{e4.16}
u_\epsilon\in W^{2,2}(B_{3/4}).
\end{equation}
Next, differentiating the equation in $\eqref{e4.6}$ with respect to $x_i$ for each $i=1,...,n$, we obtain
\begin{equation}\label{e4.17}
\Div\big((a_\epsilon(x,\nabla u_\epsilon))_{x_i}\big)=(fH_\epsilon(u_\epsilon))_{x_i}\quad\hbox{ in }{\cal D}'(B_1).
\end{equation}
Computing the derivative of $a_\epsilon(x,\nabla u_\epsilon)$ with respect to $x_i$, we get
\begin{equation}\label{e4.18}
(a_\epsilon(x,\nabla u_\epsilon))_{x_i}=\frac{\partial a_\epsilon}{\partial x_i}(x,\nabla u_\epsilon)+
D_\eta a_\epsilon(x,\nabla u_\epsilon)\cdot\nabla u_{\epsilon x_i}\quad\hbox{ a.e. in }~B_1.
\end{equation}
Using Cauchy-Schwarz inequality and the fact that $a_\epsilon$ satisfies $\eqref{1.3}$ with $\kappa=\epsilon$, we obtain
\begin{eqnarray}\label{e4.19}
|D_\eta a_\epsilon(x,\nabla u_\epsilon)\cdot\nabla u_{\epsilon x_i}|&=&\Big|\sum_{j}\frac{\partial a_{\epsilon }}{\partial\eta_j}(x,\nabla u_\epsilon)u_{\epsilon x_ix_j} \Big|\nonumber\\
&\leqslant& \sum_{j}\Big|\frac{\partial a_{\epsilon }}{\partial\eta_j}(x,\nabla u_\epsilon)\Big||u_{\epsilon x_ix_j}|\nonumber\\
&\leqslant&\Big(\sum_{k,j}\Big|\frac{\partial a_{\epsilon k}}{\partial\eta_j}(x,\nabla u_\epsilon)\Big|\Big)|\nabla u_{\epsilon x_i}|\nonumber\\
&\leqslant&c_1'(\epsilon+|\nabla u_\epsilon|^2)^{{p-2}\over 2}|\nabla u_{\epsilon x_i}|.
\end{eqnarray}
Using Cauchy-Schwarz inequality and the fact that $a_\epsilon$ satisfies $\eqref{1.3}$ with $\kappa=\epsilon$, we obtain
\begin{eqnarray}\label{e4.20}
\Big|\frac{\partial a_\epsilon}{\partial x_i}(x,\nabla u_\epsilon)\Big|&\leqslant&
c_1'(\epsilon+|\nabla u_\epsilon|^2)^{{p-2}\over 2}.
\end{eqnarray}

It follows from $\eqref{e4.16}$ and $\eqref{e4.18}$-$\eqref{e4.20}$ that we have
\begin{equation}\label{e4.21}
(a_\epsilon(x,\nabla u_\epsilon))_{x_i}\in L^2(B_{3/4}).
\end{equation}
Now, let $\varphi=\zeta^2 G(u_{\epsilon x_i})$. Since
\begin{eqnarray}\label{e4.22}
\nabla\varphi&=&\zeta^2G'(u_{\epsilon x_i})\nabla u_{\epsilon x_i}+2\zeta G(u_{\epsilon x_i})\nabla\zeta \quad\hbox{ in }~B_1,
\end{eqnarray}
we see from $\eqref{e4.16}$, $\eqref{e4.22}$ and the smoothness of $G$ and $\zeta$, that we have $\varphi\in H^1(B_{3/4})$.
Taking into account $\eqref{e4.21}$ and using $\varphi$ as a test function in $\eqref{e4.17}$, we get
\begin{eqnarray*}
&&\int_{B_1}\big(a_\epsilon(x,\nabla u_\epsilon)\big)_{x_i}\cdot\nabla\big(\zeta^2 G(u_{\epsilon x_i})\big)dx\nonumber\\
&=&-\int_{B_1}f_{x_i}H_\epsilon(u_\epsilon)\zeta^2 G(u_{\epsilon x_i})\,dx-\int_{B_{2r}(x_0)}\zeta^2fH'_\epsilon(u_{\epsilon})u_{\epsilon x_i}G(u_{\epsilon x_i})dx
\end{eqnarray*}
which leads by $\eqref{e4.18}$, $\eqref{e4.22}$ and the monotonicity of $H_\epsilon$, to
\begin{eqnarray*}
&&\int_{B_1}\Big(\frac{\partial a_\epsilon}{\partial x_i}(x,\nabla u_\epsilon)+
D_\eta a_\epsilon(x,\nabla u_\epsilon)\cdot\nabla u_{\epsilon x_i}\Big).\big(\zeta G'(u_{\epsilon x_i})\nabla u_{\epsilon x_i}+G(u_{\epsilon x_i})\nabla\zeta\big)dx\nonumber\\
&&\leq~~-\int_{B_1}f_{x_i}H_\epsilon(u_\epsilon)\zeta^2 G(u_{\epsilon x_i})dx
\end{eqnarray*}
or
\begin{eqnarray}\label{e4.23}
&&\int_{B_1}\zeta G'(u_{\epsilon x_i})D_\eta a_\epsilon(x,\nabla u_\epsilon)\cdot\nabla u_{\epsilon x_i}.\nabla u_{\epsilon x_i}dx\nonumber\\
&&\leq -\int_{B_1}G(u_{\epsilon x_i})D_\eta a_\epsilon(x,\nabla u_\epsilon)\cdot\nabla u_{\epsilon x_i}.\nabla\zeta dx\nonumber\\
&&~~-\int_{B_1}\frac{\partial a_\epsilon}{\partial x_i}(x,\nabla u_\epsilon)\cdot\nabla\big(\zeta G(u_{\epsilon x_i})\big) dx\nonumber\\
&&~~-\int_{B_1}f_{x_i}H_\epsilon(u_\epsilon)\zeta^2 G(u_{\epsilon x_i})dx.
\end{eqnarray}
Adding the inequalities from $i=1$ to $i=n$, in $\eqref{e4.23}$, we get
\begin{eqnarray}\label{e4.24}
&&\int_{B_1}\zeta \sum_{i} G'(u_{\epsilon x_i})D_\eta a_\epsilon(x,\nabla u_\epsilon)\cdot\nabla u_{\epsilon x_i}.\nabla u_{\epsilon x_i}dx\nonumber\\
&&\leq \int_{B_1}\sum_{i} |G(u_{\epsilon x_i})|.|D_\eta a_\epsilon(x,\nabla u_\epsilon)\cdot\nabla u_{\epsilon x_i}|.|\nabla\zeta| dx\nonumber\\
&&~~-\sum_{i}\int_{B_1}\frac{\partial a_\epsilon}{\partial x_i}(x,\nabla u_\epsilon).\nabla(\zeta G(u_{\epsilon x_i}))dx\nonumber\\
&&~~-\sum_{i}\int_{B_1}f_{x_i}H_\epsilon(u_\epsilon)\zeta G(u_{\epsilon x_i})dx.
\end{eqnarray}
Moreover, since $a_\epsilon$ satisfies $\eqref{1.2}$ with $\kappa=\epsilon$, we have
\begin{eqnarray}\label{e4.25}
D_\eta a_\epsilon(x,\nabla u_\epsilon)\cdot\nabla
u_{\epsilon x_i}\cdot\nabla u_{\epsilon x_i}&=&\sum_{k,j}\frac{\partial a_{\epsilon k}}{\partial\eta_j}(x,\nabla u_\epsilon)u_{\epsilon x_ix_k}u_{\epsilon x_ix_j} \nonumber\\
&\geqslant& c_0'(\epsilon+|\nabla u_\epsilon|^2)^{{p-2}\over 2}|\nabla u_{\epsilon x_i}|^2.
\end{eqnarray}
The fact, that $a_\epsilon$ satisfies also $\eqref{1.3}$ with $\kappa=\epsilon$ implies
\begin{eqnarray}\label{e4.26}
|D_\eta a_\epsilon(x,\nabla u_\epsilon)\cdot\nabla
u_{\epsilon x_i}\cdot\nabla \zeta|&\leqslant&|D_\eta a_\epsilon(x,\nabla u_\epsilon)\cdot\nabla u_{\epsilon x_i}|\cdot|\nabla \zeta|\nonumber\\
&\leqslant& c_1'(\epsilon+|\nabla u_\epsilon|^2)^{{p-2}\over 2}|\nabla u_{\epsilon x_i}||\nabla \zeta|.
\end{eqnarray}
It follows from $\eqref{e4.24}$-$\eqref{e4.26}$ that
\begin{eqnarray}\label{e4.27}
&&c_0'\int_{B_1}\zeta^2 \sum_{i} G'(u_{\epsilon x_i})(\epsilon+|\nabla u_\epsilon|^2)^{{p-2}\over 2}|\nabla u_{\epsilon x_i}|^2 dx\nonumber\\
&&\leq c_1'\int_{B_1}\sum_{i} \zeta|G(u_{\epsilon x_i})|.(\epsilon+|\nabla u_\epsilon|^2)^{{p-2}\over 2}|\nabla u_{\epsilon x_i}||\nabla \zeta| dx\nonumber\\
&&~~-\sum_{i}\int_{B_1}\frac{\partial a_\epsilon}{\partial x_i}(x,\nabla u_\epsilon).\nabla(\zeta^2 G(u_{\epsilon x_i}))dx\nonumber\\
&&~~-\sum_{i}\int_{B_1}f_{x_i}H_\epsilon(u_\epsilon)\zeta^2 G(u_{\epsilon x_i})dx.
\end{eqnarray}
To handle the second term in the right hand side of $\eqref{e4.27}$, we integrate by parts
\begin{equation}\label{e4.28}
\int_{B_1}\frac{\partial a}{\partial x_i}(x,\nabla u_\epsilon)\cdot\nabla(\zeta^2 G(u_{\epsilon x_i}))\,dx=
-\int_{B_1}\zeta^2 G(u_{\epsilon x_i})\Div\Big(\frac{\partial a}{\partial x_i}(x,\nabla u_\epsilon)\Big)\,dx.
\end{equation}
Note that we have
\begin{eqnarray}\label{e4.29}
&&\Div\Big(\frac{\partial a}{\partial x_i}(x,\nabla u_\epsilon)\Big)=\sum_{k}\frac{\partial }{\partial x_k}\Big(\frac{\partial a_k}{\partial x_i}(x,\nabla u_\epsilon)\Big)\nonumber\\
&&\quad=\sum_{k}\frac{\partial^2 a_k}{\partial x_k\partial x_i}(x,\nabla u_\epsilon)
+\sum_{k,j}\frac{\partial^2 a_k}{\partial \eta_j\partial x_i}(x,\nabla u_\epsilon)\cdot u_{\epsilon x_jx_k}.
\end{eqnarray}
Using $\eqref{e4.6}$-$\eqref{e4.7}$, we obtain
\begin{equation}\label{e4.30}
\sum_{i,k=1}^n\bigg|\frac{\partial^2 a_k}{\partial x_k\partial x_i}(x,\nabla u_\epsilon)\bigg|\leqslant c_3t_\epsilon^{p-1},
\end{equation}
\begin{equation}\label{e4.31}
\sum_{i,k,j=1}^n\bigg|\frac{\partial^2 a_k}{\partial \eta_j\partial x_i}(x,\nabla u_\epsilon)\cdot\nabla u_{\epsilon x_j}\bigg|\leq c_4t_\epsilon^{p-2}|D^2 u_\epsilon|
\end{equation}
Combining $\eqref{e4.28}$-$\eqref{e4.30}$, we get
\begin{eqnarray}\label{e4.32}
&&\sum_{i}\bigg|\int_{B_1}\frac{\partial a}{\partial x_i}(x,\nabla u_\epsilon)\cdot\nabla(\zeta^2 G(u_{\epsilon x_i}))dx\bigg|\nonumber\\
&&\quad\leqslant c_3\int_{B_1}\zeta^2 |G(u_{\epsilon x_i})|t_\epsilon^{p-1}dx+c_4\int_{B_1} \zeta^2 |G(u_{\epsilon x_i})|t_\epsilon^{p-2}
|D^2 u_\epsilon|dx.
\end{eqnarray}
Regarding the last term in the right hand side of $\eqref{e4.27}$, we have since $|G(u_{\epsilon x_i})| \leqslant|G(t_{\epsilon})|$
\begin{eqnarray}\label{e4.33}
\sum_{i}\bigg|\int_{B_1}f_{x_i}H_\epsilon(u_\epsilon)\zeta^2 G(u_{\epsilon x_i})dx\bigg|
&\leqslant&\int_{B_1}\zeta^2 \sum_{i}|f_{x_i}||G(u_{\epsilon x_i})|dx\nonumber\\
&\leqslant&\sqrt{n}\int_{B_1}\zeta^2 |G(t_{\epsilon})||\nabla f|dx.
\end{eqnarray}
Taking into account $\eqref{e4.27}$, $\eqref{e4.32}$ and $\eqref{e4.33}$, we obtain
\begin{eqnarray*}
&&c_0'\int_{B_1}\zeta^2 \sum_{i} G'(u_{\epsilon x_i})t_\epsilon^{p-2}|\nabla u_{\epsilon x_i}|^2 dx\nonumber\\
&&\leq \sqrt{n}c_1'\int_{B_1}\zeta G(t_{\epsilon})t_{\epsilon}^{p-2}|D^2 u_\epsilon||\nabla \zeta| dx\nonumber\\
&&~~+c_3\int_{B_1}\zeta^2 |G(t_{\epsilon})|t_\epsilon^{p-1}dx+c_4\int_{B_1} \zeta^2 G(t_{\epsilon}) t_\epsilon^{p-2}
|D^2 u_\epsilon|dx\nonumber\\
&&~~+\sqrt{n}\int_{B_1}\zeta^2 |G(t_{\epsilon})||\nabla f|dx.
\end{eqnarray*}
which is $\eqref{e4.15}$.
\qed

\vs 0,5cm \n \emph{Proof of Lemma $\ref{l4.1}$.} We consider $\zeta\in{\cal D}(B_{3/4})$ such that
\begin{equation*}
\begin{cases}
& 0\leqslant \zeta \leqslant 1~~
\text{ in } B_{3/4}\\
&  \zeta=1 ~~
\text{ in } B_{1/2}\\
& \displaystyle{|\nabla \zeta|\leqslant 4~~ \text{ in } B_{3/4}}.
\end{cases}
\end{equation*}
We shall consider the two possible cases.

\vs 0.2cm\n \emph{\underline{$1^{st}$ Case}}:
$1<p<2.$

\vs 0.2cm\n Let $
\displaystyle{G(t)=(\epsilon+t^2)^{{p-2}\over 2}t}$.
Then we have:

\begin{eqnarray*}G'(t)=(\epsilon+t^2)^{{p-2}\over 2}\Big[1+
{{(p-2)t^2}\over{\epsilon+t^2}}\Big]\geqslant (p-1)
(\epsilon+t^2)^{{p-2}\over 2}.
\end{eqnarray*}

\n Setting $t_\epsilon=(\epsilon+|\nabla u_\epsilon|^2)^{1/ 2}$ and
$s_\epsilon=(\epsilon+|u_{\epsilon x_i}|^2)^{1/2}$
and the fact that $0\leqslant\zeta\leqslant1$ and $|\nabla\zeta|\leqslant 4$, we get from $\eqref{e4.13}$
\begin{eqnarray}\label{e4.34}
&&\int_{B_1}\zeta^2 \sum_{i} s_\epsilon^{p-2}t_\epsilon^{p-2}|\nabla u_{\epsilon x_i}|^2 dx\leqslant
{{4c_1'\sqrt{n}+c_4}\over{c_0'(p-1)}}\int_{B_1}\zeta t_\epsilon^{p-1}t_\epsilon^{p-2}
|D^2 u_\epsilon| dx\nonumber\\
&&~~+{{c_3}\over{c_0'(p-1)}}\int_{B_1}\zeta t_\epsilon^{2(p-1)}dx+{{\sqrt{n}}\over{c_0'(p-1)}}\int_{B_1}\zeta^2 t_\epsilon^{p-1}|\nabla f|dx.
\end{eqnarray}
Using Young's inequality, we get since $\zeta=0$ outside $B_{3/4}$
\begin{eqnarray}\label{e4.35}
&&{{4c_1'\sqrt{n}+c_4}\over{c_0'(p-1)}}\int_{B_1}\zeta t_\epsilon^{p-1} t_\epsilon^{p-2}
|D^2 u_\epsilon| dx\leqslant {{(4c_1'\sqrt{n}+c_4)^2}\over{2c_0'^2(p-1)^2}}\int_{B_{3/4}} t_\epsilon^{2(p-1)}dx\nonumber\\
&&+{1\over 2}\int_{B_1}\zeta^2 [t_\epsilon^{p-2}|D^2 u_{\epsilon}|]^2 dx.
\end{eqnarray}
Taking into account $\eqref{e4.34}$-$\eqref{e4.35}$, the monotonicity of $t^{p-2}$ and the fact that $\zeta=1$ in $B_{1/2}$, we obtain
\begin{eqnarray}\label{e4.36}
&&\int_{B_{1/2}}[t_\epsilon^{p-2}|D^2 u_{\epsilon}|]^2 dx\leq {{(4c_1'\sqrt{n}+c_4)^2+2c_3c_0'(p-1)}\over{c_0'^2(p-1)^2}}\int_{B_{3/4}} t_\epsilon^{2(p-1)}dx\nonumber\\
&&~~+{{2\sqrt{n}}\over{c_0'(p-1)}}\int_{B_{3/4}} t_\epsilon^{p-1}|\nabla f|dx.
\end{eqnarray}

\vs 0.2cm\n \emph{\underline{$2^{nd}$ Case}}:
$p\geq 2.$

\vs 0.2cm\n Let $\displaystyle{G(t)=t}$.
Then we get from $\eqref{e4.15}$

\begin{eqnarray}\label{e4.37}
&&\int_{B_1}\zeta^2 t_\epsilon^{p-2}|D^2 u_\epsilon|^2 dx\leqslant {{4c_1'\sqrt{n}+c_4}\over{c_0'}}\int_{B_1}\zeta t_\epsilon t_\epsilon^{p-2}
|D^2 u_\epsilon| dx\nonumber\\
&&~~+{{c_3}\over{c_0'}}\int_{B_1}\zeta t_\epsilon^p dx+{{\sqrt{n}}\over{c_0'}}\int_{B_1}\zeta^2 t_\epsilon |\nabla f|dx.
\end{eqnarray}
Using Young's inequality, we get since $\zeta=0$ outside $B_{3/4}$
\begin{eqnarray}\label{e4.38}
&&{{4c_1'\sqrt{n}+c_4}\over{c_0'}}\int_{B_1}\zeta t_\epsilon t_\epsilon^{p-2}
|D^2 u_\epsilon| dx\leqslant {{(4c_1'\sqrt{n}+c_4)^2}\over{2c_0'^2}}\int_{B_{3/4}} t_\epsilon^p dx\nonumber\\
&&+{1\over 2}\int_{B_1}\zeta^2 t_\epsilon^{p-2}|D^2 u_{\epsilon}|^2 dx.
\end{eqnarray}
Taking into account $\eqref{e4.37}$-$\eqref{e4.38}$ and the fact that $\zeta=1$ in $B_{1/2}$, we obtain
\begin{eqnarray}\label{e4.39}
&&\int_{B_{1/2}} t_\epsilon^{p-2}|D^2 u_{\epsilon}|^2 dx\leqslant {{(4c_1'\sqrt{n}+c_4)^2+2c_3c_0'}\over{c_0'^2}}\int_{B_{3/4}} t_\epsilon^p dx\nonumber\\
&&~~+{{2\sqrt{n}}\over{c_0'}}\int_{B_{3/4}} t_\epsilon |\nabla f|dx.
\end{eqnarray}
Using the monotonicity of $t^{p-2}$ and $\eqref{e4.39}$, we get
\begin{eqnarray}\label{e4.40}
&&\int_{B_{1/2}}\Big[t_\epsilon^{p-2}|D^2 u_\epsilon|\Big]^2\,dx =\int_{B_{1/2}}t_\epsilon^{p-2}t_\epsilon^{p-2}|D^2 u_\epsilon|^2\,dx\nonumber\\
&&~~\leq\|t_\epsilon\|_{L^\infty(B_{3/4})}^{p-2}\int_{B_{1/2}}t_\epsilon^{p-2}|D^2 u_\epsilon|^2\,dx\nonumber\\
&&~~\leq {{(4c_1'\sqrt{n}+c_4)^2+2c_3c_0'(p-1)}\over{c_0'^2}}\|t_\epsilon\|_{L^\infty(B_{3/4})}^{p-2}\int_{B_{3/4}} t_\epsilon^p dx\nonumber\\
&&~~+{{2\sqrt{n}}\over{c_0'}}\|t_\epsilon\|_{L^\infty(B_{3/4})}^{2}\int_{B_{3/4}}|\nabla f| dx\nonumber\\
&&~~\leq {{(4c_1'\sqrt{n}+c_4)^2+2c_3c_0'}\over{c_0'^2}}|B_{3/4}|\|t_\epsilon\|_{L^\infty(B_{3/4})}^{2(p-1)}\nonumber\\
&&~~+{{2\sqrt{n}}\over{c_0'}}\|t_\epsilon\|_{L^\infty(B_{3/4})}^{p-1}\int_{B_{3/4}}|\nabla f| dx.
\end{eqnarray}
Combining $\eqref{e4.36}$ and $\eqref{e4.40}$, the lemma follows.
\qed
\begin{remark}\label{r4.1}
Using $\eqref{e4.3}$, $\eqref{e4.11}$, we deduce from Lemma $\ref{l4.1}$ that we have for all $\epsilon\in(0,1)$
\begin{eqnarray*}
&&E_\epsilon(1/2,u_\epsilon)~~\leqslant {{(4c_1'\sqrt{n}+c_4)^2+2c_3c_0'}\over{c_0'^2\min(1,p-1)^2}}|B_{3/4}|M_1^{2(p-1)}\nonumber\\
&&~~+{{2\sqrt{n}}\over{c_0'\min(1,p-1)}}M_1^{p-1}\int_{B_{3/4}}|\nabla f| dx\leqslant C_1,
\end{eqnarray*}
where $C_1$ is a positive constant depending on $n$, $p$, $c_0'$, $c_1'$, $c_3$, $c_4$,
$M_1$ and $C_0$.
\end{remark}

\vs 0.2cm Now we estimate $E_\epsilon(r,u_\epsilon)$.
\begin{lemma}\label{l4.3}
If the conditions of Lemma $\ref{l4.1}$ are satisfied, then we have for all $\epsilon\in(0,1)$ and $r\in(0,1/2)$
\begin{eqnarray*}
&&E_\epsilon(r,u_\epsilon)\leqslant{{3^n(4c_1'\sqrt{n}+c_4)^2+2c_3c_0'(p-1)}\over{2^{n+2}c_0'^2(p-1)^2r^2}}|B_{3/4}|\|t_{\epsilon r}\|_{L^\infty(B_{3/4})}^{2(p-1)}\nonumber\\
&&~~+{{\sqrt{n}}\over{c_0'(p-1)|B_{1/2}|2^{n-1}r^n}}\|t_{\epsilon r}\|_{L^\infty(B_{3/4})}^{p-1}\int_{B_{3/4}}|\nabla f(2rx)| dx.
\end{eqnarray*}
\end{lemma}

\vs 0.3cm\n \emph{Proof.} Let $\epsilon\in(0,1)$ and $r\in (0,{1\over 2})$. We consider the function $u_{\epsilon r}(x)=
\displaystyle{ {u_\epsilon( 2rx) }\over {2r}}$ defined  in $B_1$.  By definition, $u_{\epsilon r}$ is the
unique solution of the problem
\begin{equation*}
\begin{cases} &  u_{\epsilon r}-u_r\in W_0^{1,p}(B_{1\over{2r}})\\
&  \textrm{div}(a_{\epsilon r}(x,\nabla u_{\epsilon r}))= f_r H_\epsilon(u_{\epsilon r}) \quad \text{in }\quad
B_{1\over{2r}},
\end{cases}
\end{equation*}
where $u_r(x)=\displaystyle{ {u( 2rx) }\over {2r}}$, $f_r(x)=2rf(2rx)$, and $a_{\epsilon r}(x,\eta)=a_{\epsilon}(2rx,\eta)$ are functions
defined in $B_{1\over{2r}}$, with $u_r$ a solution of the following class of problems
\begin{equation*}\mathcal{F}_{a_r(\cdot)}:
\begin{cases}
&u_r\in W^{1,p}(B_1)\cap C^{1,\alpha}(B_1),\\
&\Div\big(a_r(x,\nabla u_r(x))\big)=f_r(x)\hbox{ in }\{u_r>0\}\cap B_1,\\
&0\leq u_r\leq M_1\hbox{ in }B_1,\\
&0\in\partial\{u_r>0\},
\end{cases}
\end{equation*}
and where $M_1$ is the positive number in $\eqref{e4.10}$.

\vs0.2cm Indeed, first it is obvious that $0\in\partial\{u_r>0\}$, $u_r\in W^{1,p}(B_1)\cap C^{1,\alpha}(B_1)$,
and that we have from $\eqref{e4.10}$
\begin{equation}\label{e4.41}
\|\nabla u_r\|_{L^\infty(B_{3/4})}=\|\nabla u\|_{L^\infty(B_{3r/2})}\leq M_1,\qquad\forall u\in\mathcal{F}_{A_r(\cdot)},
\end{equation}
Moreover, we have
\begin{eqnarray*}
&&\Div\big(a_r(x,\nabla u_r)\big)(x)=\Div\big(a(2rx,\nabla u(2rx))\big)\nonumber\\
&&\quad =2rf(2rx)=f_r(x)\quad \text{in }~~\{u(rx)>0\}=\{u_r(x)>0\},
\end{eqnarray*}
and from $\eqref{e4.41}$, we have since $u_r(0)=0$
\begin{eqnarray*}
0\leq u_r(x)=\int_0^1{{d}\over{dt}} u_r(tx)\,dt=\int_0^1\nabla u(2trx)\cdot x\,dt \leq M_1\quad \forall x\in \overline{B}_1.
\end{eqnarray*}
Next, we observe that $f_r$ satisfies $\eqref{e4.1}$-$\eqref{e4.2}$ with the constants $2r\Lambda$
and $2rC_0$, $a_{\epsilon r}(x,\eta)$ satisfies $\eqref{1.2}$-$\eqref{1.3}$ with $\kappa=\epsilon$ and $\eqref{e4.3}$-$\eqref{e4.5}$
with the constants $c_0'$, $c_1'$, $2rc_2$, $c_3$, $c_4$ and $p$. Obviously, the constants $2r\Lambda$, $2rC_0$, $2rc_2$, $4r^2c_3$
and $2rc_4$ are bounded above respectively by $\Lambda$, $C_0$, $c_2$, $c_3$ and $c_4$ for $r\in\big(0,\frac{1}{4}\big)$. Setting
$t_{\epsilon r}=(\epsilon+|\nabla u_\epsilon(2rx)|^2)^{1/2}$,
and applying Lemma $\ref{l4.1}$ to $u_{\epsilon r}$, we obtain
\begin{eqnarray*}
&&E_\epsilon({1/2},u_{\epsilon r})\leqslant{{3^n(4c_1'\sqrt{n}+c_4)^2+2c_3c_0'(p-1)}\over{2^nc_0'^2(p-1)^2}}|B_{3/4}|\|t_{\epsilon r}\|_{L^\infty(B_{3/4})}^{2(p-1)}\nonumber\\
&&~~+{{2\sqrt{n}}\over{c_0'(p-1)|B_{1/2}|}}\|t_{\epsilon r}\|_{L^\infty(B_{3/4})}^{p-1}\int_{B_{3/4}}|\nabla f_r| dx
\end{eqnarray*}
or
\begin{eqnarray}\label{e4.42}
&&E_\epsilon({1/2},u_{\epsilon r})\leqslant{{3^n(4c_1'\sqrt{n}+c_4)^2+2c_3c_0'(p-1)}\over{2^nc_0'^2(p-1)^2}}|B_{3/4}|\|t_{\epsilon r}\|_{L^\infty(B_{3/4})}^{2(p-1)}\nonumber\\
&&~~+{{8r^2\sqrt{n}}\over{c_0'(p-1)|B_{1/2}|}}\|t_{\epsilon r}\|_{L^\infty(B_{3/4})}^{p-1}\int_{B_{3/4}}|\nabla f(2rx)| dx.
\end{eqnarray}
Note that
\begin{eqnarray}\label{e4.43}
E_\epsilon(r,u_\epsilon)&=&{1\over{|B_r|}}\int_{{B_r}}\big[(\epsilon+|\nabla u_\epsilon(x)|^2)^{{p-2}\over 2}|D^2 u_\epsilon(x)|\big]^2\,dx\nonumber\\
&=&{1\over{|B_{1/2}|}} \int_{B_{1/2}}\big[(\epsilon+|\nabla u_\epsilon(2rx)|^2)^{{p-2}\over 2} |D^2 u_\epsilon(2rx)|\big]^2\,dx\nonumber\\
&=&{1\over{4r^2}}{1\over{|B_{1/2}|}} \int_{B_1/2}\big[(\epsilon+|\nabla u_\epsilon(2rx)|^2)^{{p-2}\over 2} |2rD^2 u_\epsilon(2rx)|\big]^2\,dx\nonumber\\
&=&{{E_\epsilon({1/2},u_{\epsilon r})}\over{4r^2}}.
\end{eqnarray}
Taking into account $\eqref{e4.42}$-$\eqref{e4.43}$ and $\eqref{e4.14}$, we get
\begin{eqnarray*}
&&E_\epsilon(r,u_\epsilon)\leqslant{{3^n(4c_1'\sqrt{n}+c_4)^2+2c_3c_0'}\over{2^{n+2}c_0'^2\min(1,p-1)^2r^2}}|B_{3/4}|\|t_{\epsilon r}\|_{L^\infty(B_{3/4})}^{2(p-1)}\nonumber\\
&&~~+{{2\sqrt{n}}\over{c_0'\min(1,p-1)|B_{1/2}|}}\|t_{\epsilon r}\|_{L^\infty(B_{3/4})}^{p-1}\int_{B_{3/4}}|\nabla f(2rx)| dx.
\end{eqnarray*}
or
\begin{eqnarray*}
&&E_\epsilon(r,u_\epsilon)\leqslant{{3^n(4c_1'\sqrt{n}+c_4)^2+2c_3c_0'(p-1)}\over{2^{n+2}c_0'^2\min(1,p-1)^2r^2}}|B_{3/4}|\|t_{\epsilon r}\|_{L^\infty(B_{3/4})}^{2(p-1)}\nonumber\\
&&~~+{{\sqrt{n}}\over{c_0'\min(1,p-1)|B_{1/2}|2^{n-1}r^n}}\|t_{\epsilon r}\|_{L^\infty(B_{3/4})}^{p-1}\int_{B_{3/4}}|\nabla f(x)| dx
\end{eqnarray*}
which completes the proof of the lemma.
\qed

\subsection{Hausdorff measure of the free boundary for $\kappa=0$}

In this section we extend the local
finiteness of the $(n-1)$-dimensional Hausdorff measure of the free boundary
for a heterogeneous operator of $p-$Laplacian type. This property was obtained only in
homogeneous cases, for the $p-$Obstacle problem in \cite{[C]} with $p=2$, in \cite{[LS]} for $p>2$,
and more generally for the $A-$Obstacle problem \cite{[CLR]} that includes the case $1<p<\infty$ (see also \cite{[ZZ12]}).
The new difficulty is in the control of the additional $x$ dependence of the quasilinear coefficients $a_i=a_i(x,\eta)$, requiring the additional assumptions (4.3) and (4.4).

\begin{theorem}\label{t4.1} Assume that
$a$ satisfies $\eqref{1.2}$ with $\kappa=0$ and $\eqref{e4.3}$, $\eqref{e4.4}$, and that $f$ is nonnegative and locally
bounded in $\Omega$, $\nabla f\in\mathcal{M}_{\loc}^n(\Omega)$.
Then for each $\lambda>0$, the free boundary of the $a(\cdot)-$obstacle problem (P) is locally of
finite $(n-1)$-dimensional Hausdorff measure in $\{f(x)>\lambda\}$.
\end{theorem}

\vs 0,5cm\n Due to the local character of Theorem $\ref{t4.1}$, it is enough to
give the proofs for the solutions of the class of problems $\mathcal{F}_{a(\cdot)}$,
which for convenience, we state in the next two theorems. For this purpose,
we assume that $f$ satisfies
\begin{equation}\label{e4.44}
0<\lambda\leq f\quad\text{a.e. in }~~B_1.
\end{equation}

\begin{theorem}\label{t4.2}
Assume that $f$ satisfies $\eqref{e4.1}$-$\eqref{e4.2}$ and $\eqref{e4.44}$, and that $a$ satisfies
$\eqref{1.2}$ (with $\kappa=0$) and $\eqref{e4.3}$-$\eqref{e4.4}$. Then there exists a constant $C$
depending only on $n$, $p$, $c_0$, $c_1$, $c_2$, $c_3$, $c_4$, $\lambda$, $\Lambda$,
$M_0$ and $C_0$ such that for each $u\in\mathcal{F}_{a(\cdot)}$,
for each $x_0\in\partial\{u>0\}\cap B_{1/2}$ and $r\in\big(0,\frac{1}{4}\big)$, we have
$$
\mathcal{H}^{n-1}(\partial\{u>0\}\cap B_r(x_0))\leq Cr^{n-1}.
$$
\end{theorem}
In order to prove the theorem, we need two lemmas.
\begin{lemma}\label{l4.4}
Assume that $a$ satisfies $\eqref{1.2}$ (with $\kappa=0$) and $\eqref{e4.3}$-$\eqref{e4.4}$, and that
$f$ satisfies $\eqref{e4.2}$, $\eqref{e4.44}$. Then we have
\begin{eqnarray*}
H_{\epsilon}^2(u_\epsilon)
&\leq&{{2c_1'^2}\over{\lambda^2}}\big[t_\epsilon^{p-2}|D^2u_\epsilon|\big]^2+\frac{8c_2^2}{\lambda^2}t_\epsilon^{2(p-1)}.
\end{eqnarray*}
\end{lemma}

\emph{Proof}. Since $\lambda H_\epsilon(u_\epsilon)\leqslant f H_\epsilon(u_\epsilon)$, we get
by recalling $\eqref{e4.7}$ and the fact that $a_\epsilon$ satisfies $\eqref{1.3}$ with $\kappa=\epsilon$
\begin{eqnarray*}
\lambda H_\epsilon(u_\epsilon)&\leq&div\big(a_\epsilon(x,\nabla u_\epsilon)\big)=\sum_{i=1}^n\frac{\partial a^i_\epsilon}{\partial x_i}(x,\nabla u_\epsilon)
+\sum_{i,j=1}^n \frac{\partial a^i_\epsilon}{\partial \eta_j}(x,\nabla u_\epsilon) u_{\epsilon x_ix_j}\nonumber\\
&\leq&\sum_{i=1}^n\bigg|\frac{\partial a^i_\epsilon}{\partial x_i}(x,\nabla u_\epsilon)\bigg|+\sum_{i,j=1}^n \bigg|\frac{\partial a^i_\epsilon}{\partial \eta_j}(x,\nabla u_\epsilon)\bigg|
 |u_{\epsilon x_ix_j}|\nonumber\\
&\leq&\sum_{i=1}^n\bigg|\frac{\partial a^i_\epsilon}{\partial x_i}(x,\nabla u_\epsilon)\bigg|+\bigg(\sum_{i,j=1}^n \bigg|\frac{\partial a^i_\epsilon}{\partial \eta_j}(x,\nabla u_\epsilon)\bigg|\bigg)|D^2u_\epsilon|\nonumber\\
&\leq&2c_2\big(\epsilon+|\nabla u_\epsilon|^2\big)^{{p-1}\over2}
+c_1'\big(\epsilon+|\nabla u_\epsilon|^2\big)^{{p-2}\over2}|D^2u_\epsilon|\nonumber\\
&=&2c_2 t_\epsilon^{p-1}+c_1't_\epsilon^{p-2}|D^2u_\epsilon|).
\end{eqnarray*}
It follows that
\begin{eqnarray*}
\lambda^2H_{\epsilon}^2(u_\epsilon)&\leq&8c_2^2t_\epsilon^{2(p-1)}+
2c_1'^2t_\epsilon^{(p-2)}|D^2u_\epsilon|^2
 \end{eqnarray*}
or
\begin{eqnarray*}
H_{\epsilon}^2(u_\epsilon)
&\leq&{{2c_1'^2}\over{\lambda^2}}\big[t_\epsilon^{p-2}|D^2u_\epsilon|\big]^2+\frac{8c_2^2}{\lambda^2}t_\epsilon^{2(p-1)}.
\end{eqnarray*}

\qed

\begin{lemma}\label{l4.5}
Assume that $f$ satisfies $\eqref{e4.1}$-$\eqref{e4.2}$, $\eqref{e4.44}$. Assume also that
$a$ satisfies $\eqref{1.2}$ (with $\kappa=0$) and $\eqref{e4.3}$-$\eqref{e4.4}$. Then there
exists a positive constant $C$ depending only on $n$, $p$, $c_0$, $c_1$, $c_2$, $\lambda$,
$M_0$ and $C_0$ such that for each $u\in\mathcal{F}_{A(\cdot)}$, any $\delta\in(0,1)$ and
$r\in(0,1/4)$ with $B_{2r}(x_0)\subset B_1$ and $x_0\in B_{1/2}\cap\partial\{u>0\}$, we have
$$
\mathcal{L}^n(O_\delta\cap B_r(x_0)\cap\{u>0\})\leq C\delta r^{n-1},
$$
where $O_\delta=\{|\nabla u|<\delta^{\frac{1}{p-1}}\}\cap B_{1/2}$.
\end{lemma}

\n\emph{Proof} Let $u\in\mathcal{F}_{a(\cdot)}$, $x_0\in B_{1/2}\cap\partial\{u>0\}$, $\delta\in(0,1)$
and $r\in(0,1/4)$ with $B_{2r}(x_0)\subset B_1$.
\vs 0.2cm\n For each $\epsilon\in(0,1)$ and
$\eta=2^{p-1}\delta$, we consider the function
\begin{equation*}
G(t)=\begin{cases} (\epsilon+\eta^{2\over{p-1}})^{{p-2}\over2}\eta^{1\over{p-1}} ~~
&\text{ if }~ t> \eta^{1\over{p-1}}\\
\max\big((\epsilon+t^2)^{{p-2}\over2},(\epsilon+\eta^{2\over{p-1}})^{{p-2}\over2}\big)t ~~
&\text{ if } ~|t|\leqslant \eta^{1\over{p-1}}\\
-(\epsilon+\eta^{2\over{p-1}})^{{p-2}\over2}\eta^{1\over{p-1}} ~~ &\text{ if }~
t<-\eta^{1\over{p-1}}.
\end{cases}
\end{equation*}

\n We have $G(0)=0$, and $G$ is Lipschitz continuous with
\begin{equation}\label{e4.45}
G'(t)=\begin{cases} (\epsilon+t^2)^{{p-2}\over2}\Big[1+{{(p-2)t^2}\over{\epsilon+t^2}}\Big]\chi_{\{|t|<
\eta^{1\over{p-1}}\}} ~~
&\text{ if }~ p \leq 2\\
(\epsilon+\eta^{2\over{p-1}})^{{p-2}\over2}\chi_{\{|t|<
\eta^{1\over{p-1}}\}} ~~
&\text{ if } ~p>2.
\end{cases}
\end{equation}
We also have
\begin{eqnarray}\label{e4.46}
|G(t)|\leqslant(\epsilon+\eta^{2\over{p-1}})^{{p-1}\over 2}\quad\forall t.
\end{eqnarray}
We denote by $u_\epsilon$ the solution of the problem $\eqref{e4.6}$ and we consider a function $\zeta\in\mathcal{D}(B_{2r}(x_0))$
such that
\begin{eqnarray}\label{e4.47}
0\leqslant\zeta\leqslant 1\hbox{ in }B_{2r}(x_0),\quad\zeta=1\hbox{ in }B_r(x_0),\quad|\nabla\zeta|\leq\frac{2}{r}\hbox{ in }B_{2r}(x_0),
\end{eqnarray}
First we have from $\eqref{e4.15}$
\begin{eqnarray}\label{e4.48}
&&c_0'\int_{B_1}\zeta^2 \sum_{i} G'(u_{\epsilon x_i})t_\epsilon^{p-2}|\nabla u_{\epsilon x_i}|^2 dx\nonumber\\
&&\leq \sqrt{n}c_1'\int_{B_1}\zeta G(t_{\epsilon})t_{\epsilon}^{p-2}|D^2 u_\epsilon||\nabla \zeta| dx\nonumber\\
&&~~+c_3\int_{B_1}\zeta^2 G(t_{\epsilon})t_\epsilon^{p-1}dx+c_4\int_{B_1} \zeta^2 G(t_{\epsilon}) t_\epsilon^{p-2}
|D^2 u_\epsilon|dx\nonumber\\
&&~~+\sqrt{n}\int_{B_1}\zeta^2 G(t_{\epsilon})|\nabla f|dx.
\end{eqnarray}
Taking into account $\eqref{e4.45}$-$\eqref{e4.47}$ and the fact that
$\{|\nabla u_{\epsilon}|<\eta^{1\over{p-1}}\}\subset\{|u_{\epsilon x_i}|<\eta^{1\over{p-1}}\}$,
we obtain from $\eqref{e4.48}$
\begin{eqnarray}\label{e4.49}
&&\int_{{B_r(x_0)\cap\{|\nabla u_{\epsilon}|<\eta^{1\over{p-1}}\}}}[t_\epsilon^{p-2}|D^2 u_\epsilon|]^2 dx\nonumber\\
&&~~\leqslant  {{2\sqrt{n}c_1'}\over {\min(1,p-1)rc_0'}}\big(\epsilon+\eta^{\frac{2}{p-1}}\big)^{\frac{p-1}{2}}\int_{B_{2r}(x_0)} t_{\epsilon}^{p-2}|D^2 u_\epsilon| dx\nonumber\\
&&~~+{{c_3}\over {\min(1,p-1)c_0'}}\big(\epsilon+\eta^{\frac{2}{p-1}}\big)^{\frac{p-1}{2}}\int_{B_{2r}(x_0)} t_\epsilon^{p-1}dx\nonumber\\
&&~~+{{c_4}\over {\min(1,p-1)c_0'}}\big(\epsilon+\eta^{\frac{2}{p-1}}\big)^{\frac{p-1}{2}}\int_{B_{2r}(x_0)}  t_\epsilon^{p-2}
|D^2 u_\epsilon|dx\nonumber\\
&&~~+{{\sqrt{n}}\over {\min(1,p-1)c_0'}}\big(\epsilon+\eta^{\frac{2}{p-1}}\big)^{\frac{p-1}{2}}\int_{B_{2r}(x_0)}|\nabla f|dx.
\end{eqnarray}
Using the Schwarz inequality and Remark $\ref{r4.1}$, we get
\begin{eqnarray}\label{e4.50}
&&\int_{B_r(x_0)}[t_\epsilon^{p-2}|D^2 u_\epsilon|]dx\nonumber\\
&&\quad\leqslant \Big(\int_{B_{2r}(x_0)}1^2dx\Big)^{1/2}.\Big(\int_{B_{2r}(x_0)}\big[t_\epsilon^{p-2}|D^2u_\epsilon|\big]^2dx\Big)^{1/2}\nonumber\\
&&\quad\leqslant |B_{2r}|^{1/2}\Big(|B_{2r}(x_0)|E_\epsilon(2r,u_\epsilon))^{1/2}\nonumber\\
&&\quad\leqslant |B_{2r}|(E_\epsilon(1/2,u_\epsilon))^{1/2}
\leqslant \sqrt{C_2}|B_{2r}|.
\end{eqnarray}
Combining $\eqref{e4.49}$-$\eqref{e4.50}$, we get since $\epsilon,\eta\in(0,1)$
\begin{eqnarray*}
&&\int_{{B_r(x_0)\cap\{|\nabla u_{\epsilon}|<\eta^{1\over{p-1}}\}}}[t_\epsilon^{p-2}|D^2 u_\epsilon|]^2 dx\leqslant  {{2\sqrt{n}c_1'}\over {\min(1,p-1)rc_0'}}\big(\epsilon+\eta^{\frac{2}{p-1}}\big)^{\frac{p-1}{2}}\sqrt{C_2}|B_{2r}|\nonumber\\
&&~~+{{c_3}\over {\min(1,p-1)c_0'}}\big(\epsilon+\eta^{\frac{2}{p-1}}\big)^{\frac{p-1}{2}}\int_{B_{2r}(x_0)} t_\epsilon^{p-1}dx\nonumber\\
&&~~+{{c_4}\over {\min(1,p-1)c_0'}}\big(\epsilon+\eta^{\frac{2}{p-1}}\big)^{\frac{p-1}{2}}\sqrt{C_2}|B_{2r}|\nonumber\\
&&~~+{{\sqrt{n}}\over {\min(1,p-1)c_0'}}\big(\epsilon+\eta^{\frac{2}{p-1}}\big)^{\frac{p-1}{2}}\int_{B_{2r}(x_0)}|\nabla f|dx.
\end{eqnarray*}
or
\begin{eqnarray}\label{e4.51}
&&\int_{B_r(x_0)\cap\{|\nabla u_{\epsilon}|<\eta^{1\over{p-1}}\}}[t_\epsilon^{p-2}|D^2u_\epsilon|]^2 dx\nonumber\\
&&\leqslant {{\big(\epsilon+\eta^{\frac{2}{p-1}}\big)^{\frac{p-1}{2}}} \over{\min(1,p-1)c_0'}}
\Big[\sqrt{C_2}\Big({{2\sqrt{n}c_1'}\over {r}}+c_4\Big)|B_{2r}|+c_3\int_{B_{2r}(x_0)} t_\epsilon^{p-1}dx
+\sqrt{n}\int_{B_{2r}(x_0)}|\nabla f|dx\Big]\nonumber\\
&&\leqslant {{\big(\epsilon+\eta^{\frac{2}{p-1}}\big)^{\frac{p-1}{2}}} \over{\min(1,p-1)c_0'}}
\Big[\sqrt{C_2}\Big({{2\sqrt{n}c_1'}\over {r}}+c_4\Big)|B_{2r}|+c_3|B_{2r}| M_1^{p-1}
+\sqrt{n}C_0 r^{n-1}\Big]\nonumber\\
&&={{\big(\epsilon+\eta^{\frac{2}{p-1}}\big)^{\frac{p-1}{2}}} \over{\min(1,p-1)c_0'}}
\Big[2\sqrt{n}c_1'\sqrt{C_2}+\sqrt{n}C_0 +r|B_2|(c_3 M_1^{p-1}+c_4)\Big]r^{n-1}.
\end{eqnarray}
Since $O_{\delta}\subset \{|\nabla u_\epsilon|<\eta^{1\over{p-1}}\}$ and
\begin{eqnarray*}
\int_{B_r(x_0)\cap O_{\delta}}t_\epsilon^{2(p-1)}dx&\leqslant&\int_{B_r(x_0)\cap \{|\nabla u_\epsilon|<\eta^{1\over{p-1}}\}}t_\epsilon^{2(p-1)}dx\\
&\leqslant&\big(\epsilon+\eta^{\frac{2}{p-1}}\big)^{p-1}|B_1|r^n,
\end{eqnarray*}
we get from $\eqref{e4.51}$ by using $\eqref{e4.11}$
\begin{eqnarray}\label{e4.52}
&&\int_{B_r(x_0)\cap O_{\delta}}H_{\epsilon}^2(u_\epsilon)\leqslant
{{2c_1'^2}\over{\lambda^2}}\int_{B_r(x_0)\cap O_{\delta}}\big[t_\epsilon^{p-2}|D^2u_\epsilon|\big]^2dx+\frac{8c_2^2}{\lambda^2}\int_{B_r(x_0)\cap O_{\delta}}t_\epsilon^{2(p-1)} dx\nonumber\\
&&\leqslant {{2c_1'^2}\over{\lambda^2}}\int_{B_r(x_0)\cap \{|\nabla u_{\epsilon}|<\eta^{1\over{p-1}}\}}\big[t_\epsilon^{p-2}|D^2u_\epsilon|\big]^2dx +\frac{8c_2^2}{\lambda^2}\int_{B_r(x_0)\cap \{|\nabla u_{\epsilon}|<\eta^{1\over{p-1}}\}}t_\epsilon^{2(p-1)} dx\nonumber\\
&&\quad\leqslant \frac{8c_2^2}{\lambda^2}\big(\epsilon+\eta^{\frac{2}{p-1}}\big)^{p-1}|B_1|r^n\nonumber\\
&&\quad +{{2c_1'^2\big(\epsilon+\eta^{\frac{2}{p-1}}\big)^{\frac{p-1}{2}}} \over{\lambda^2\min(1,p-1)c_0'}}
\Big[2\sqrt{n}c_1'\sqrt{C_2}+\sqrt{n}C_0 +r|B_2|(c_3 M_1^{p-1}+c_4)\Big]r^{n-1}.\nonumber\\
&&
\end{eqnarray}
Letting $\epsilon\rightarrow 0$ in $\eqref{e4.52}$, we obtain
\begin{eqnarray*}
&&{\cal L}^n(O_\delta\cap B_r(x_0)\cap \{u>0\})
\leqslant \frac{8c_2^2}{\lambda^2}\eta^2|B_1|r^n\nonumber\\
&&\quad +{{2c_1^2} \over{\lambda^2\min(1,p-1)c_0}}\eta
\Big[2\sqrt{n}c_1\sqrt{C_2}+\sqrt{n}C_0 +r|B_2|(c_3 M_1^{p-1}+c_4)\Big]r^{n-1},
\end{eqnarray*}
which leads to
\begin{eqnarray*}
{\cal L}^n(O_\delta\cap B_r(x_0)\cap \{u>0\})\leqslant C\delta r^{n-1},
\end{eqnarray*}
where $C$ is a positive constant depending on $n$, $p$, $c_0$, $c_1$, $c_3$, $c_4$, $\lambda$,
$M_1$ and $C_0$.

\qed

\vs 0,5cm\n \emph{Proof of Theorem $\ref{t4.2}$.} Let $r\in\big(0,\frac{1}{4}\big)$, $B_r(x_0)\subset B_1$
with $x_0\in\partial\{u>0\}\cap B_{1/2}$ and $\delta>0$.
Let $E$ be a subset of $\mathbb{R}^n$ and $s\in[0,\infty)$. The $s$-dimensional Hausdorff measure of $E$ is defined by
$$
\mathcal{H}^s(E)=\lim_{\delta\rightarrow0}H_\delta^s(E)=\sup_{\delta>0}H_\delta^s(E),
$$
where
$$
H_\delta^s(E)=\inf\bigg\{\sum_{j=1}^\infty \alpha(s)\bigg(\frac{diam(C_j)}{2}\bigg)^s\,\big|\,E\subset\bigcup_{j=1}^\infty C_j,\,diam(C_j)\leq\delta\bigg\},
$$
$\displaystyle{\alpha(s)=\frac{\pi^{s/2}}{\Gamma(s/2+1)}}$, $\displaystyle{\Gamma(s)=\int_0^\infty e^{-t}t^{s-1}\,dt}$ for $s>0$ is the Gamma function.\\

We argue as in the proof of Theorem 1.5 of \cite{[CLR]}. More precisely, let $E=\partial\{u>0\}\cap B_r(x_0)$
and denote by $\displaystyle{\big(B_\delta(x_i)\big)_{i\in I}}$ a finite covering of $E$, with $x_i\in\partial\{u>0\}$ and $P(n)$ maximum overlapping.\\
From the proof of Theorem $\ref{t3.1}$, there exists a constant $c_0$ such that
$$
\forall i\in I\quad\exists y_i\in B_\delta(x_i)~:\quad B_{c_0\delta}(y_i)\subset B_\delta(x_i)\cap\{u>0\}\cap O_\delta.
$$
We deduce from Lemma $\ref{l4.5}$ that
\begin{eqnarray*}
\sum_{i\in I}\mathcal{L}^n(B_1)c_0^n\delta^n&=&\sum_{i\in I}\mathcal{L}^n(B_{c_0\delta}(y_i))\leq\sum_{i\in I}\mathcal{L}^n(B_\delta(x_i)\cap\{u>0\}\cap O_\delta)\nonumber\\
&\leq&P(n)\mathcal{L}^n(B_\delta(x_i)\cap\{u>0\}\cap O_\delta)\leq P(n)C\delta r^{n-1},
\end{eqnarray*}
where $C>0$ is the constant from Lemma $\ref{l4.5}$. This leads to
$$
\sum_{i\in I}\alpha(n-1)\bigg(\frac{diam(B_\delta(x_i))}{2}\bigg)^{n-1}\leq\frac{\alpha(n-1)}
{\mathcal{L}^n(B_1)c_0^n}P(n)Cr^{n-1}=\overline{C}r^{n-1},
$$
so
$$
H_\delta^{n-1}(\partial\{u>0\}\cap B_r(x_0))\leq\overline{C}r^{n-1}.
$$
Letting $\delta\rightarrow0$, we obtain
$$
\mathcal{H}^{n-1}(\partial\{u>0\}\cap B_r(x_0))\leq\overline{C}r^{n-1}.
$$

\qed

\section{Second order regularity for $\kappa>0$}\label{s5}
Here we extend a second order regularity result to non degenerate operators similar to the one established in
\cite{[CL11]} in the $p(x)-$Laplacian framework.\\

For $\kappa>0$, we consider the family of problems
\begin{equation}\label{e5.1}
\left\{
  \begin{array}{ll}
   \Div\big(a(x,\nabla u)\big)=f \qquad  &\textrm{in }\Omega, \\
   u=g  \quad &\hbox{on } \partial\Omega,
  \end{array}
\right.
\end{equation}
where $f\in L^\infty(\Omega)$ and $g\in W^{1,p(\cdot)}(\Omega)$.

We will assume that $a(x,\eta)$ satisfies $\eqref{1.2}$-{\eqref{1.4}} and that $p$ satisfies $\eqref{1.1}$, $\eqref{2.1}$.
By a solution of $\eqref{e5.1}$ we mean a function $u\in W^{1,p(\cdot)}(\Omega)$ satisfying
$$
\left\{
  \begin{array}{ll}
   \displaystyle{\int_\Omega} a(x,\nabla u)\cdot\nabla\xi\,dx=-\int_\Omega f\xi\,dx,\quad\forall\xi\in W_0^{1,p(\cdot)}(\Omega), \\
   u-g\in W_0^{1,p(\cdot)}(\Omega).
  \end{array}
\right.
$$

By the classical theory of monotone operators, we know that problem $\eqref{e5.1}$
has a unique solution. Moreover, the solution of $\eqref{e5.1}$ is known to have $C^{1,\alpha}_{\loc}$ regularity
\cite{[F]}. In this section, we are concerned with second order regularity. This kind of regularity is classical
for $p$-Laplace type operators with $p$ constant. We refer, for example to \cite{[G1]} Theorem 8.1, Theorem 6.5 of
\cite{[LU]} and \cite{[T]}. To establish the $W^{2,2}_{\loc}$ estimate, we shall apply the method based on the
difference quotients $\Delta_h$ as in the above references, and \cite{[CL11]} in the case of the $p(x)$-Laplacian.

We will denote by $\|v\|_\infty$ the usual norm of functions in $L^\infty(\Omega)$. Note that, recalling Remark
$\ref{r1.1}$ also by Theorem 4.1 of \cite{[FZ]}, since $f\in L^\infty(\Omega)$, the solution of $\eqref{e5.1}$
is locally bounded i.e. $u\in L^\infty_{\loc}(\Omega)$. We shall assume here that $u\in L^\infty(\Omega)$.
More precisely, there exists a positive constant $M$
such that $\|u\|_\infty\leq M$. Since $p$ is Lipschitz continuous, then for each $\Omega'\subset\Omega$, we have from
\cite{[F]} that
$$
\|u\|_{C^{1,\alpha}(\Omega')}\leq C,
$$
where $\alpha=\alpha(n,p_-,p_+,L,M,\|f\|_\infty)$ and $C=C(n,p_-,p_+,L,M,\|f\|_\infty,d(\Omega',\Omega))$
are positive real numbers.

First, let us define for each $h\neq0$ and each vector $e_s$ $(s=1,\ldots,n)$ of the canonical basis of $\mathbb{R}^n$ ,
the difference quotient of a function $\varphi$ by
$$
\Delta_{s,h}\varphi(x):=\frac{\varphi(x+he_s)-\varphi(x)}{h}.
$$
The function $\Delta_{s,h}\varphi$ is well defined on the set $\Delta_{s,h}\Omega:=\{x\in\Omega\,/\,x+he_s\in\Omega\}$, which contains the set $\Omega_{|h|}:=\{x\in\Omega\,/\,d(x,\partial\Omega)>|h|\}$.

Since $W^{1,p(\cdot)}(\Omega)\hookrightarrow W^{1,p_-}(\Omega)\hookrightarrow W^{1,1}(\Omega)$, some properties in \cite{[G1]} (p. 263)
of difference quotients are still valid. In particular we have
\begin{itemize}
  \item If $\varphi\in W^{1,1}(\Omega)$, then $\Delta_{s,h}\varphi\in W^{1,1}(\Omega)$, and we have
  $\nabla(\Delta_{s,h}\varphi)=\Delta_{s,h}(\nabla\varphi)$.
  \item $\Delta_{s,h}(\varphi_1\varphi_2)(x)=\varphi_1(x+he_s)\Delta_{s,h}\varphi_2(x)+\varphi_2(x)\Delta_{s,h}\varphi_1(x)$
  for functions $\varphi_1$ and $\varphi_2$ defined in $\Omega$.
  \item If at least one of the functions $\varphi_1$ or $\varphi_2$ has support contained in $\Omega_{|h|}$, then we have
  $$
  \int_\Omega\varphi_1\Delta_{s,h}\varphi_2=-\int_\Omega\varphi_2\Delta_{s,h}\varphi_1.
  $$
  \item If $w\in W^{1,m}(B_{4R})$ $(m\geq1)$ and $\zeta^2\Delta_{s,h}w\in W^{1,1}(B_{3R})$ for $\zeta\in\mathcal{D}(B_{3R})$,
  we have (\cite{[G1]}, Lemma 8.1) for $|h|<R$ and some constant $c(n)$,
      \begin{align*}
        &\|\Delta_{s,h}w\|_{L^m(B_{2R})}\leq c(n)\|D_sw\|_{L^m(B_{3R})}\\
        &\|\Delta_{s,-h}(\zeta^2\Delta_{s,h}w)\|_{L^1(B_{2R})}\leq c(n)\|D_s(\zeta^2\Delta_{s,h}w)\|_{L^1(B_{3R})}.
      \end{align*}
\end{itemize}
For simplicity, we will drop the dependence on $s$ and write $\Delta_h$ for $\Delta_{s,h}$, etc.
Here is the main result of this section.
\begin{theorem}\label{t5.1}
If $u$ is the solution of $\eqref{e5.1}$ with $\kappa>0$, then $u\in W^{2,2}_{\loc}(\Omega)$.
\end{theorem}

\begin{proof}
Let $R>0$ be such that the open ball $B_{2R}(x_0)$ satisfies $\overline{B}_{2R}(x_0)\subset\Omega$.
We consider a function $\xi\in\mathcal{D}(B_{2R}(x_0))$ such that
$$
\left\{
  \begin{array}{ll}
  0\leq\xi\leq1,\textrm{ in }B_{2R},\qquad\xi=1\textrm{ in }B_R(x_0),\\
  |\nabla\xi|^2+|D^2\xi|\leq\frac{c}{R^2}\textrm{ in }B_{2R}(x_0).
  \end{array}
\right.
$$
Then $\Delta_{s,-h}(\xi^2\Delta_{s,h}u)$ is a test function for $\eqref{e5.1}$, and we have
$$
\int_\Omega a(x,\nabla u)\cdot\nabla\big(\Delta_{-h}(\xi^2\Delta_hu)\big)\,dx=-\int_\Omega f\Delta_{-h}(\xi^2\Delta_hu)\,dx,
$$
which leads to
\begin{equation}\label{e5.2}
\int_\Omega\Delta_h a(x,\nabla u)\cdot\big(\xi^2\nabla(\Delta_hu)+2\xi\Delta_hu\nabla\xi\big)\,dx=-\int_\Omega f\Delta_{-h}(\xi^2\Delta_hu)\,dx.
\end{equation}
Let $x_h:=x+he_s$ and write
\begin{equation}\label{e5.3}
\Delta_ha\big(x,\nabla u(x)\big)=\frac{1}{h}\big[a\big(x_h,\nabla u(x_h)\big)-a\big(x,\nabla u(x)\big)\big]:=U+V,
\end{equation}
where
\begin{align*}
&U:=\frac{1}{h}\big[a\big(x_h,\nabla u(x_h)\big)-a\big(x,\nabla u(x_h)\big)\big],\\
&V:=\frac{1}{h}\big[a\big(x,\nabla u(x_h)\big)-a\big(x,\nabla u(x)\big)\big].
\end{align*}
It follows then from ${\eqref{e5.2}}$ and $\eqref{e5.3}$ that
\begin{eqnarray}\label{e5.4}
  &&\int_\Omega\xi^2V\cdot\nabla(\Delta_hu)=-\int_\Omega\xi^2U\cdot\nabla(\Delta_hu)-\int_\Omega2\xi(\Delta_hu)U\cdot\nabla\xi\,dx\nonumber\\
  &&~-\int_\Omega2\xi(\Delta_hu)V\cdot\nabla\xi\,dx-\int_\Omega f\Delta_{-h}(\xi^2\Delta_hu)\,dx.
\end{eqnarray}
Writing $\nabla u(x_h)=\big(\nabla u+h\Delta_h(\nabla u)\big)(x)$ and setting $\theta_t=\big(\nabla u+th\Delta_h(\nabla u)\big)(x)$, we obtain
\begin{eqnarray*}
  V&=&\frac{1}{h}\int_0^1\frac{d}{dt}\bigg[a\big(x,(\nabla u+th\Delta_h(\nabla u))(x)\big)\bigg]\,dt\nonumber\\
  &=&\int_0^1\nabla_\eta a\big(x,(\nabla u+th\Delta_h(\nabla u))(x)\big)\cdot\Delta_h(\nabla u)\,dt.
\end{eqnarray*}
It follows then
$$
V\cdot\nabla (\Delta_hu)=\int_0^1\nabla_\eta a\big(x,(\nabla u+th\Delta_h(\nabla u))(x)\big)\cdot\Delta_h(\nabla u)\nabla(\Delta_hu)\,dt.
$$
Multiplying the last equality by $\xi^2$ and integrating with respect to $x$ over $\Omega$, we obtain
\begin{eqnarray*}
  &&\int_\Omega\xi^2V\cdot\nabla(\Delta_hu)\,dx\\
  &&=\int_\Omega\bigg[\xi^2\int_0^1\nabla_\eta a\big(x,(\nabla u+th\Delta_h(\nabla u))(x)\big)\cdot\Delta_h(\nabla u)\nabla(\Delta_hu)\,dt\bigg]\,dx:=I\nonumber.
\end{eqnarray*}
Using $\eqref{1.2}$ one has
\begin{eqnarray}\label{e5.5}
I\geq c_0\int_\Omega\bigg[\xi^2|\nabla(\Delta_hu)|^2\int_0^1\big(\kappa+|\theta_t|^2\big)^\frac{{p(x)-2}}{2}\,dt\bigg]\,dx\geq 0.
\end{eqnarray}
Next, we write
\begin{eqnarray*}
U&=&\frac{1}{h}\big\{a(x_h,\nabla u(x_h))-a(x,\nabla u(x_h))\big\}\nonumber\\
&=&\frac{1}{h}\int_0^1\frac{d}{dt}a\big(x+the_s,\nabla u(x_h)\big)\,dt\nonumber\\
&=&\int_0^1\nabla_xa\big(x+the_s,\nabla u(x_h)\big).e_s\,dt.
\end{eqnarray*}
Recalling $\eqref{1.4}$, the fact that $u\in C^{1,\alpha}_{\loc}(\Omega)$ and that $p(\cdot)$ is Lipschitz
continuous in $\Omega$, we easily deduce from the above equality, that for some positive constant $C$, one has
\begin{equation}\label{e5.6}
|U|\leq C.
\end{equation}
Hence, by Young's inequality we get for $\nu>0$
\begin{eqnarray}\label{e5.7}
&&\bigg|\int_\Omega\xi^2U\cdot\nabla(\Delta_hu)\,dx\bigg|\leq\int_\Omega\xi^2|U||\nabla(\Delta_hu)|\,dx\nonumber\\
&\leq&\nu\int_\Omega\xi^2|\nabla(\Delta_hu)|^2\,dx+\frac{C^2}{4\nu}\int_\Omega\xi^2\,dx\nonumber\\
&\leq&\nu\int_\Omega\xi^2|\nabla(\Delta_hu)|^2\,dx+\frac{C^2}{4\nu}|B_{2R}|.
\end{eqnarray}
Using $\eqref{e5.7}$, we estimate the second term in the right hand side of $\eqref{e5.4}$ as follows
\begin{eqnarray}\label{e5.8}
\bigg|-2\int_\Omega\xi(\Delta_hu)U\cdot\nabla\xi\bigg|&\leq&\frac{2C c^{1/2}}{R}\int_{B_{2R}}|\Delta_hu|\nonumber\\
&\leq&\frac{2C c^{1/2}c(n)}{R}\int_{B_{3R}}|\nabla u|\,dx\leq C'.
\end{eqnarray}
In order to estimate the third term in the right hand side of $\eqref{e5.4}$, we need to estimate $V$.
For this purpose, referring to the above definition of $V$ (after the equality $\eqref{e5.4}$)
and using $\eqref{1.3}$, we have

\begin{eqnarray*}
|V|&\leq& c_1\int_0^1\big|\big(\nabla u+th\Delta_h(\nabla u)\big)(x)\big|^{p(x)-2}|\Delta_h(\nabla u)|\,dt\nonumber\\
&\leq&c_1W(x)|\Delta_h(\nabla u)|,
\end{eqnarray*}
where $W(x)=\displaystyle{\int_0^1\big(\kappa+|\theta_t|^2\big)^{\frac{p(x)-2}{2}}}\,dt$.

Now since $u\in C^{1,\alpha}(\overline{B}_{2R})$, it is easy to see that there exist two positive constants
$l_\kappa$ and $L_\kappa$, depending on $\kappa$, such that $l_\kappa\leq W(x)\leq L_\kappa$.
Moreover we have $|\Delta_hu|\leq\|\nabla u\|_{L^\infty(B_{3R})}$. Therefore it follows by Young's inequality
that for every $\mu>0$
\begin{eqnarray}\label{e5.9}
&&\bigg|\int_\Omega2\xi V\nabla\xi\Delta_hu\,dx\bigg|\leq 2c_1L_\kappa\int_\Omega\xi|\Delta_h(\nabla u)||\nabla\xi||\Delta_h u|\,dx\nonumber\\
&\leq&\mu\int_\Omega\xi^2|\Delta_h(\nabla u)|^2\,dx+\frac{4c_1^2 L^2_\kappa}{\mu}\int_\Omega|\nabla\xi|^2|\Delta_hu|^2\,dx.
\end{eqnarray}
Using again Young's inequality, for $\lambda>0$ for the last term in the right hand side of $\eqref{e5.4}$,
we have, since $f\in L^\infty(\Omega)$
\begin{eqnarray}\label{e5.10}
&&\bigg|\int_\Omega f\Delta_{-h}(\xi^2\Delta_hu)\,dx\bigg|\leq\|f\|_\infty\int_\Omega|\Delta_{-h}(\xi^2\Delta_hu)|\nonumber\\
&\leq& c(n)\|f\|_\infty\int_\Omega|\nabla(\xi^2\Delta_hu)|\,dx\nonumber\\
&\leq&c(n)\|f\|_\infty\int_\Omega\bigg[\xi^2|\nabla(\Delta_hu)|+2\xi|\nabla\xi||\Delta_hu|\bigg]\,dx\nonumber\\
&\leq&\lambda\int_\Omega\xi^2|\nabla(\Delta_h)|^2+c^2(n)\|f\|^2_\infty\frac{|B_{2R}|}{4\lambda}\nonumber\\
&&+\frac{2c^{1/2}}{R}c^2(n)\|f\|_\infty\int_{B_{2R}}|\nabla u|\,dx.
\end{eqnarray}
Hence, choosing $\nu=\mu=\lambda=\frac{l_\kappa}{3}$, we obtain
from $\eqref{e5.4}$-$\eqref{e5.10}$ for a positive constant $C=C(n,\kappa,p_-,p_+,L,R,\|f\|_\infty)$
$$
\l_\kappa\int_\Omega\xi^2|\nabla(\Delta_hu)|^2\,dx\leq C,
$$
which leads to
$$
\int_{B_R}|\nabla(\Delta_hu)|^2\,dx\leq C/\l_\kappa.
$$
Letting $h\rightarrow0$, we obtain the desired result \cite{[G1]}, Lemma 8.9.
\end{proof}

Due to Proposition 2.1 $iii)$, as an immediate consequence, we also have this local
second order regularity result for the obstacle problem.
\begin{corollary}\label{c5.1}
Under the assumptions of Theorem 5.1, namely for $\kappa>0$, if $u$ is the solution of the obstacle
problem $(P)$, then $u\in W^{2,2}_{\loc}(\Omega)\cap C^{1,\alpha}(\Omega)$ for some $\alpha>0$.
\end{corollary}

\section{$\mathcal{H}^{n-1}$-measure of the free boundary for $\kappa>0$}\label{s6}

The main result of this section is the local finiteness of the $\mathcal{H}^{n-1}$-measure of the
essential free boundary. It is known that the free boundary locally has finite $\mathcal{H}^{n-1}$-measure
for several homogeneous operators: the $p-$Obstacle problem, \cite{[C]} for $p=2$ and \cite{[LS]} for $p>2$,
and more generally for a homogeneous operator of $p-$Laplacian type \cite{[ZZ12]}, and for the $A-$Obstacle
problem \cite{[CLR]} that also includes the $p-$Laplacian ($1<p<\infty$).

It turns out, that the heterogeneous case is much more delicate in the $p(x)$ framework, as we now treat in
this section for $\kappa>0$. In this case we show that at least the essential free boundary has locally finite
$\mathcal{H}^{n-1}$-measure. We use the bounded variation approach of Br\'{e}zis and Kinderlehrer (see \cite{[BK]}
or \cite{[KS]}) by showing that $Au\in BV_{\loc}(\Omega)$, which implies, for a nondegenerating forcing $f$, that the
set $\{u>0\}$ has locally finite perimeter. Hence $\partial_e\{u>0\}$ has locally finite $\mathcal{H}^{n-1}$-measure
(see, for example \cite{[EG]}), where $\partial_eE$ is the essential boundary of $E$. As an important consequence,
by a well-known result of De Giorgi (see $\cite{[G]}$, page 54), the free boundary may be written, up to a possible
singular set of $\|\nabla \chi_{\{u>0\}}\|$-measure zero, as a countable union of $C^1$ hypersurfaces.
\begin{definition}\label{d6.1}
Let $\omega\subset\Omega$. We say that the function $g\in L^1(\omega)$ is of bounded variation in $\omega$ and write
$g\in BV(\omega)$, if there exists a positive constant $C$ such that
$$
\bigg|\int_\omega g\zeta_{x_i}\,dx\bigg|\leq C\|\zeta\|_{L^\infty(\Omega)},\,\hbox{ for }1\leq i\leq n\, \hbox{ and }\,\zeta\in C^\infty(\Omega).
$$
If $g\in BV(\omega)$, we define its variation $V_\omega g$ as follows:
$$
V_\omega g=\sup\big\{\sum_{i=1}^n\int_\omega g\zeta_{ix_i}\,dx;\,\,\zeta_i\in C^\infty(\Omega),\,\,|\zeta|\leq1\big\}.
$$
\end{definition}
In this section we will assume additionally that
\begin{equation}\label{e6.1}
\sum_{i,j=1}^n\bigg|\frac{\partial^2a_i}{\partial x_i\partial x_j}(x,\eta)\bigg|\leq c_3\big(\kappa+|\eta|^2\big)^{\frac{p(x)-1}{2}}\big(1+\big|\ln\big(\kappa+|\eta|^2\big)^{\frac{1}{2}}\big|\big)\big|\ln\big(\kappa+|\eta|^2\big)^{\frac{1}{2}}\big|,
\end{equation}
\begin{equation}\label{e6.2}
\sum_{i,j,k=1}^n\bigg|\frac{\partial^2a_k}{\partial\eta_j\partial x_i}(x,\eta)\bigg|\leq c_4\big(\kappa+|\eta|^2\big)^{\frac{p(x)-2}{2}}\big(1+\big|\ln\big(\kappa+|\eta|^2\big)^{\frac{1}{2}}\big|\big),
\end{equation}
for some positive constants $c_3$, $c_4$.

We shall also assume that $f$ satisfies $\eqref{3.1}$, and $\nabla f\in\mathcal{M}_{\loc}^n(\Omega)$
(Morrey space, \cite{[MZ97]}), which means that there exists a positive constant $C_0$ such that
\begin{equation}\label{e6.3}
\int_{B_r}|\nabla f|\,dx\leq C_0r^{n-1}, \textrm{ for any }B_r\subset\subset\Omega.
\end{equation}
In particular, $\eqref{e6.3}$ is satisfied, if $f\in C^{0,1}(\overline{\Omega})$.

\begin{theorem}\label{t6.1}
Assume that $p(\cdot)$ satisfies $\eqref{2.1}$, $f$ satisfies $\eqref{3.1}$, $\eqref{e6.3}$, and that $\eqref{1.2}$-$\eqref{1.4}$,
$\eqref{e6.1}$, $\eqref{e6.2}$ hold with $\kappa>0$. Then $Au=\Div(a(x,\nabla u))\in BV_{\loc}(\Omega)$.
\end{theorem}

\begin{proof} Let $B_r(x_0)$ such that $B_{2r}(x_0)\subset\subset \Omega$. For simplicity, we drop the dependence on $x_0$. 
We will prove that $V_{B_r}(Au)\leq c$ for some positive constant $c$. To do that,
we select an approximation to sign(t), that is, a sequence of smooth functions $\gamma_\delta(t)$, $\delta>0$ satisfying
\begin{eqnarray*}
  &&|\gamma_\delta(t)|\leq1,\,\,\gamma_\delta'(t)\geq0,\, t\in\mathbb{R},\\
  && \gamma_\delta(0)=0,\,\,\lim_{\delta\rightarrow0}\gamma_\delta(t)=\hbox{sign}(t).
\end{eqnarray*}

\n We also consider a cutoff function $\zeta\in C_0^\infty(B_{2r})$
such that $\zeta=1$ in $B_r$ and $0\leq\zeta\leq1$ in $B_{2r}$.

\n We introduce for $\epsilon\in(0,1)$, the unique solution of the following approximating problem
\begin{equation}\label{e6.4}
\left\{
  \begin{array}{ll}
   u_\epsilon-g\in W_0^{1,p(\cdot)}(\Omega), \\
   \Div\big(a(x,\nabla u_\epsilon)\big)=fH_\epsilon(u_\epsilon)\,\,\textrm{ in }\,\Omega,
  \end{array}
\right.
\end{equation}
where $g$ is the same as in $(P)$, and where $H_\epsilon$ is as in Section 4.

\n First, we observe \cite{[F]} that there exist two constants $\alpha\in(0,1)$
and $M_1>1$ independent of $\epsilon$ such that
$u_\epsilon\in C^{1,\alpha}_{loc}(\Omega)$ and 
\begin{equation}\label{e6.5}
\|u_\epsilon\|_{C^{1,\alpha}(\overline{B}_{2r})}\leqslant M_1.
\end{equation}

\n Moreover, we know from Theorem 5.1 that we have for a positive constant $M_2$ 
independent of $\epsilon$ 
\begin{equation}\label{e6.6}
||u_\epsilon||_{W^{2,2}(B_{2r})}\leqslant M_2,
\end{equation}
and in particular, we have for a positive constant $c_5$ independent of $\epsilon$
\begin{equation}\label{e6.7}
\int_{B_{2r}}|D^2 u_\epsilon|dx\leqslant c_5.
\end{equation}
\vs 0.2cm \n We shall first prove that there exists a positive constant $c_6$
independent of $\epsilon$ and $\delta$ such that we have for each $k=1,...,n$
\begin{equation}\label{e6.8}
\int_{B_r}\zeta\gamma_\delta(u_{\epsilon x_k})(Au_\epsilon)_{x_k} dx\leq c_6.
\end{equation}

\n Integrating by parts, we get
\begin{eqnarray}\label{e6.9}
&&\int_{B_{2r}}\zeta\gamma_\delta(u_{\epsilon x_k})(Au_\epsilon)_{x_k}dx
=-\int_{B_{2r}}(a(x,\nabla u_\epsilon))_{x_k}.\nabla(\zeta\gamma_\delta(u_{\epsilon x_k}))dx\nonumber\\
&&=-\int_{B_{2r}}\bigg(\frac{\partial a}{\partial x_k}(x,\nabla u_\epsilon)+
D_\eta a(x,\nabla u_\epsilon)\cdot\nabla u_{\epsilon x_k}\bigg).\nabla(\zeta\gamma_\delta(u_{\epsilon x_k}))dx\nonumber\\
&&=-\int_{B_{2r}}\frac{\partial a}{\partial x_k}(x,\nabla u_\epsilon).\nabla(\zeta\gamma_\delta(u_{\epsilon x_k}))dx
-\int_{B_{2r}} \gamma_\delta(u_{\epsilon x_k}) D_\eta a(x,\nabla u_\epsilon)\cdot\nabla u_{\epsilon x_k}.\nabla\zeta dx\nonumber\\
&&~~-\int_{B_{2r}} \zeta \gamma_\delta'(u_{\epsilon x_k}) D_\eta a(x,\nabla u_\epsilon)\cdot\nabla u_{\epsilon x_k}.\nabla u_{\epsilon x_k} dx.
\end{eqnarray}

\n Since $a$ satisfies $\eqref{1.2}$, we have for a.e. $x\in B_{2r}$
\begin{eqnarray}\label{e6.10}
D_\eta a(x,\nabla u_\epsilon)\cdot\nabla u_{\epsilon x_k}\cdot\nabla u_{\epsilon x_k}
&\geqslant& c_0(\kappa+|\nabla u_\epsilon|^2)^{{p(x)-2}\over 2}|\nabla u_{\epsilon x_k}|^2.
\end{eqnarray}
The fact that $a$ satisfies also $\eqref{1.3}$, implies that for a.e. $x\in B_{2r}$
\begin{eqnarray}\label{e6.11}
|D_\eta a(x,\nabla u_\epsilon)\cdot\nabla
u_{\epsilon x_k}\cdot\nabla \zeta|&\leqslant&|D_\eta a(x,\nabla u_\epsilon)\cdot\nabla u_{\epsilon x_k}|\cdot|\nabla \zeta|\nonumber\\
&\leqslant& c_1(\kappa+|\nabla u_\epsilon|^2)^{{p(x)-2}\over 2}|\nabla u_{\epsilon x_k}||\nabla \zeta|.
\end{eqnarray}
Using the fact that $\zeta$ and $\gamma_\delta'$ are nonnegative and that $|\gamma_\delta|\leqslant 1$, we deduce from $\eqref{e6.9}$-$\eqref{e6.11}$ that
\begin{eqnarray}\label{e6.12}
&&\int_{B_{2r}}\zeta\gamma_\delta(u_{\epsilon x_k})(Au_\epsilon)_{x_k}dx\leqslant
-\int_{B_{2r}}\frac{\partial a}{\partial x_k}(x,\nabla u_\epsilon).\nabla(\zeta\gamma_\delta(u_{\epsilon x_k}))dx\nonumber\\
&&\quad +c_1|\nabla \zeta|_\infty\int_{B_{2r}} (\kappa+|\nabla u_\epsilon|^2)^{{p(x)-2}\over 2}|\nabla u_{\epsilon x_k}| dx=J_1+J_2.
\end{eqnarray}

\n Using $\eqref{e6.5}$ and $\eqref{e6.7}$, we see that
\begin{eqnarray}\label{e6.13}
J_2&\leqslant& c_1|\nabla \zeta|_\infty (\kappa+M_1^2)^{{p_+-2}\over 2}\int_{B_{2r}} |\nabla u_{\epsilon x_k}| dx\nonumber\\
&\leqslant& c_1c_5|\nabla \zeta|_\infty (\kappa+M_1^2)^{{p_+-2}\over 2}=c_7.
\end{eqnarray}

\n To handle $J_1$, we integrate by parts
\begin{equation}\label{e6.14}
J_1=\int_{B_{2r}}\zeta \gamma_\delta(u_{\epsilon x_k})\Div\Big(\frac{\partial a}{\partial x_k}(x,\nabla u_\epsilon)\Big)dx.
\end{equation}
Note that we have
\begin{eqnarray}\label{e6.15}
&&\Div\Big(\frac{\partial a}{\partial x_k}(x,\nabla u_\epsilon)\Big)=\sum_{i}\frac{\partial }{\partial x_i}\Big(\frac{\partial a_i}{\partial x_k}(x,\nabla u_\epsilon)\Big)\nonumber\\
&&\quad=\sum_{i}\frac{\partial^2 a_i}{\partial x_i\partial x_k}(x,\nabla u_\epsilon)
+\sum_{i,j}\frac{\partial^2 a_i}{\partial \eta_j\partial x_k}(x,\nabla u_\epsilon)\cdot u_{\epsilon x_jx_i}.
\end{eqnarray}
Using $\eqref{e6.1}$-$\eqref{e6.2}$, we obtain
\begin{eqnarray}\label{e6.16}
&&\sum_{i=1}^n\bigg|\frac{\partial^2 a_i}{\partial x_i\partial x_k}(x,\nabla u_\epsilon)\bigg|\leqslant
c_3\big(\kappa+|\nabla u_\epsilon|^2\big)^{\frac{p(x)-1}{2}}\big(1+\big|\ln\big(\kappa+|\nabla u_\epsilon|^2\big)^{\frac{1}{2}}\big|\big)\big|\ln\big(\kappa+|\nabla u_\epsilon|^2\big)^{\frac{1}{2}}\big|\nonumber\\
&&\quad \leqslant c_3c(\kappa,p_+,M_1)=c_8,
\end{eqnarray}
\begin{eqnarray}\label{e6.17}
&&\sum_{i,j=1}^n\bigg|\frac{\partial^2 a_i}{\partial \eta_j\partial x_k}(x,\nabla u_\epsilon)\cdot u_{\epsilon x_jx_i}\bigg|\leq c_4
\big(\kappa+|\nabla u_\epsilon|^2\big)^{\frac{p(x)-2}{2}}\big(1+\big|\ln\big(\kappa+|\nabla u_\epsilon|^2\big)^{\frac{1}{2}}\big|\big)|D^2 u_\epsilon|\nonumber\\
&&\quad \leqslant c_4c(\kappa,p_+,M_1)|D^2 u_\epsilon|=c_9 |D^2 u_\epsilon|.
\end{eqnarray}

\n Combining $\eqref{e6.14}$-$\eqref{e6.17}$ and using the fact that $|\zeta \gamma_\delta(u_{\epsilon x_k})|\leqslant 1$, we get
\begin{eqnarray}\label{e6.18}
&&J_1\leqslant \int_{B_{2r}}\bigg|\Div\Big(\frac{\partial a}{\partial x_k}(x,\nabla u_\epsilon)\Big)\bigg|dx\nonumber\\
&&\quad\leqslant \int_{B_{2r}} \sum_{i}\bigg|\frac{\partial^2 a_i}{\partial x_i\partial x_k}(x,\nabla u_\epsilon)\bigg|dx
+ \int_{B_{2r}} \sum_{i,j}\bigg|\frac{\partial^2 a_i}{\partial \eta_j\partial x_k}(x,\nabla u_\epsilon)\cdot u_{\epsilon x_jx_i}\bigg|dx\nonumber\\
&&\quad\leqslant c_8|B_{2r}|+c_9\int_{B_{2r}} |D^2 u_\epsilon|dx\leqslant c_8|B_{2r}|+c_5c_9=c_{10}.
\end{eqnarray}

\n We deduce from $\eqref{e6.12}$, $\eqref{e6.13}$, and $\eqref{e6.18}$ that (6.8) holds for $c_6=c_7+c_{10}$.

\vs0.2cm\n Now differentiating $\eqref{e6.4}$ with respect to $x_k$ for $k=1,...,n$,  we obtain
\begin{equation}\label{e6.19}
(Au_\epsilon)_{x_k}= f_{x_k} H_\epsilon(u_\epsilon)+fH'_\epsilon(u_\epsilon)u_{\epsilon x_k}.
\end{equation}

\n Multiplying $\eqref{e6.19}$ by $\zeta\gamma_\delta(u_{\epsilon x_k})$ and integrating over $B_{2r}$, we get
\begin{eqnarray*}
\int_{B_{2r}}\zeta\gamma_\delta(u_{\epsilon x_k})(Au_\epsilon)_{x_k}dx&=&\int_{B_{2r}}\zeta\gamma_\delta(u_{\epsilon x_k}) f_{x_k} H_\epsilon(u_\epsilon)dx\\
&&+\int_{B_{2r}} f\zeta\gamma_\delta(u_{\epsilon x_k})H'_\epsilon(u_\epsilon)u_{\epsilon x_k}dx
\end{eqnarray*}
which leads by taking into account $\eqref{e6.3}$ and $\eqref{e6.8}$ and
using the fact that $|\zeta \gamma_\delta(u_{\epsilon x_k})H_\epsilon(u_\epsilon)|\leqslant 1$ to
\begin{eqnarray}\label{e6.20}
&&\int_{B_{2r}} f\zeta\gamma_\delta(u_{\epsilon x_k})H'_\epsilon(u_\epsilon)u_{\epsilon x_k}dx=\int_{B_{2r}}\zeta\gamma_\delta(u_{\epsilon x_k})(Au_\epsilon)_{x_k}dx\nonumber\\
&&~~-\int_{B_{2r}}\zeta\gamma_\delta(u_{\epsilon x_k}) f_{x_k} H_\epsilon(u_\epsilon)dx\leqslant
 c_6+\int_{B_{2r}}|f_{x_k}|dx\nonumber\\
&&\quad\leqslant c_6+C_0(2r)^{n-1}=c_{11}.
\end{eqnarray}

\n On the other hand, since $H'_\epsilon(u_\epsilon)\gamma_\delta(u_{\epsilon x_k})u_{\epsilon x_k}$ is a nonnegative function, we have
\begin{equation*}
\lim_{\delta\rightarrow0}H'_\epsilon(u_\epsilon)\gamma_\delta(u_{\epsilon x_k})u_{\epsilon x_k}=|(H_\epsilon(u_\epsilon))_{x_k}|\,\,
\hbox{ a.e. in }~B_{2r},
\end{equation*}
which leads by the bounded convergence theorem to
\begin{equation}\label{e6.21}
\int_{B_{2r}}\zeta f|(H_\epsilon(u_\epsilon))_{x_k}|\,dx\leqslant c_{11}.
\end{equation}

\n Multiplying again $\eqref{e6.19}$ by $\zeta$ and integrating over $B_{2r}$, we get 
by taking into account the fact that $|\zeta H_\epsilon(u_\epsilon)|\leqslant 1$ and $\eqref{e6.3}$ 
\begin{eqnarray}\label{e6.22}
\int_{B_{2r}}\zeta|(Au_\epsilon)_{x_k}|dx&\leqslant&\int_{B_{2r}}(|\zeta H_\epsilon(u_\epsilon)| |f_{x_k}| 
+ f\zeta|H_\epsilon(u_\epsilon)|_{x_k}) dx\nonumber\\
&\leqslant&\int_{B_{2r}}|f_{x_k}| dx+ \int_{B_{2r}} \zeta f|(H_\epsilon(u_\epsilon))_{x_k}| dx\nonumber\\
&\leqslant&C_0(2r)^{n-1}+c_{11}=c_{12}.
\end{eqnarray}

\n Since $\zeta$ is nonnegative and $\zeta=1$ in $B_r$, we deduce from $\eqref{e6.22}$ that
\begin{eqnarray*}
\int_{B_{r}} |(Au_\epsilon)_{x_k}|dx\leqslant c_{12},~~\forall k=1,...,n.
\end{eqnarray*}
Hence we obtain $Au_\epsilon\in W^{1,1}_{\loc}(B_r)$ uniformly.
Finally we observe from $\eqref{e6.5}$-$\eqref{e6.6}$ that the approximating sequence of solutions $u_\epsilon$ converges in
$W^{2,2}_{\loc}(\Omega)-weakly$ and in $C^{1,\beta}(\Omega)$, for some $\beta>0$, to the
solution $u$ of the obstacle problem and consequently also $Au_\epsilon \rightarrow Au$
in $L^{2}_{\loc}(\Omega)-weakly$ which concludes the proof of the theorem.
\end{proof}

As a consequence, we get the main result of this section:
\begin{theorem}\label{t6.2}
Assume that $p$ satisfies $\eqref{1.1}$, $\eqref{2.1}$, $f$ satisfies $\eqref{3.1}$ 
and $\eqref{e6.3}$, and that $a$ satisfies $\eqref{1.2}$-$\eqref{1.4}$ and $\eqref{e6.1}$, $\eqref{e6.2}$,
and additionally $\sum_{i=1}^n\frac{\partial a_i}{\partial x_i}(x,0)= 0$. 
Then the essential free boundary of problem $(P)$ has locally finite $\mathcal{H}^{n-1}$-measure.
\end{theorem}
\begin{proof} From Proposition 2.2 $iii)$ we know that the solution $u$ of the obstacle problem satisfies $f\chi_{\{u>0\}}\leqslant A u\leqslant f\textit{ a.e. in }\Omega$ and as a consequence of its regularity given by Corollary 5.1, $u\in W^{2,2}_{\loc}(\Omega)\cap C^{1,\alpha}(\Omega)$. Therefore
$$
Au=\sum_{i=1}^n\frac{\partial a_i}{\partial x_i}(x,\nabla u)+ \sum_{i,j=1}^n \frac{\partial a_i}{\partial \eta_j}(x, \nabla u) \frac{\partial ^2u}{\partial x_i\partial x_j} = 0,
$$
for a.e. $x\in\{u=0\}$ and consequently we have
$$
Au=f\chi_{\{u>0\}}\,\,\hbox{a.e. in }\Omega.
$$
By Theorem 6.1 and the assumptions on $f$ we conclude
$$
\displaystyle{\frac{Au}{f}}=\chi_{\{u>0\}}\in BV_{\loc}(\Omega).
$$
This means that the set $\{u>0\}$ has locally finite perimeter, which immediately implies (see, for example \cite{[EG]}, page 204) that $\mathcal{H}^{n-1}(\partial_e\{u>0\}\cap B_r)<\infty$, for any $r\in(0,R)$.
\end{proof}

\begin{remark}\label{r6.1}
We recall that the essential free boundary $\partial_e\{u>0\}\cap B_r$ (or the measure-theoretic
free boundary) consists of points which have positive upper $n$-dimensional Lebesgue densities
with respect to the two subsets $\{u>0\}\cap B_r$ and $\{u=0\}\cap B_r$. The singular part
$\Sigma_0=(\partial\{u>0\}\setminus \partial_e\{u>0\})\cap B_r$ has null perimeter, i.e., the set
$\Sigma_0$ of free boundary points which are not on the essential free boundary has
$\|\nabla \chi_{\{u>0\}}\|$-measure zero, but its fine structure in the general case  is unknown.
However a characterization of the singular set of the obstacle problem may be given, but is essentially
restricted to the case of the Laplacian operator (see $\cite{[PSU]}$, Chapter 7).
\end{remark}

\vs 0,3cm \noindent\emph{Acknowledgments.} The authors  thank John Andersson for his observations on a
previous version of this manuscript that led to its correction and to this final version. The first
and second authors are grateful for the excellent research facilities at
the Fields Institute during their visits at this institute.

\vs 0,5cm

\end{document}